\documentclass[a4paper,10pt]{article}
\usepackage{mathtext}
\usepackage[T1,T2A]{fontenc}
\usepackage[cp1251]{inputenc}
\usepackage[english]{babel}
\usepackage{amsmath}
\usepackage{amsfonts}
\usepackage{amssymb}
\usepackage{mathrsfs}
\usepackage{amsthm}
\usepackage{enumerate}
\usepackage{graphicx}
\graphicspath{}
\DeclareGraphicsExtensions{.pdf,.png,.jpg}

\usepackage{color}
\usepackage{euscript}

\usepackage{cite}

\newtheorem{Le}{Lemma}[section]
\newtheorem{Def}{Definition}[section]
\newtheorem{St}[Le]{Proposition}
\newtheorem{Th}{Theorem}[section]
\newtheorem{Cor}[Le]{Corollary}
\newtheorem{Rem}[Le]{Remark}

\newtheorem{Que}{Question}

\numberwithin{equation}{section}

\newcommand{\eps}{\varepsilon}

\DeclareMathOperator{\supp}{supp}

\DeclareMathOperator{\diam}{diam}
\DeclareMathOperator{\dist}{dist}

\title{Uncertainty Principle for distributions with Fourier transform in $L_{p,q}(\mathbb{R}^d)$.}
\author{Nikita Dobronravov\footnote{Supported by the Russian Science Foundation grant no 24-71-10011}}
\begin{document}
	\maketitle
	\begin{abstract}
		A version of the Uncertainty Principle says: There does not exist a non zero function in $L_p(\mathbb{R}^d)$ if its Fourier transform is supported by a set of finite $\alpha$-Hausdorff measure with $\alpha<2d/p$. This UP does not hold at the endpoint $\alpha=2d/p$. We find the sharp form of the UP in the limit case. 
		We prove that there exists a non-zero function in the Lorentz space $L_{p,q}(\mathbb{R}^d)$ such that  its Fourier transform is supported by a set of zero $(\frac{2d}{p},\beta)$-Netrusov--Hausdorff capacity if and only if $\beta>\frac{q}{2(q-1)}$.
	\end{abstract}
	
	\section{Introducion}
	The Uncertainty Principle (UP) in mathematical analysis is a family of facts that state: Both function and its Fourier transform cannot be simultaneously small (see~\cite{HavJor}). Denote the Hausdorff measure of dimension $\alpha$ by $\mathcal{H}_\alpha$. The following theorem is yet another manifestation of the Uncertainty Principle.
	\begin{Th}\label{nonend} 
		Let $S\subset \mathbb{R}^d$ be a compact set such that $\mathcal{H}_{\alpha}(S)<\infty$. Let $\zeta$ be a tempered distribution such that $\supp\zeta\subset S$ and $\hat{\zeta}\in L_p(\mathbb{R}^d)$ for some $p<\frac{2d}{\alpha}$. Then $\zeta=0$.
	\end{Th}
	
		In the article~\cite{Mynotyet} it was proved that this UP does not hold at the endpoint.
		\begin{Th}\label{res}
			Let $2<p<\infty$. There exists a compact set $S\subset \mathbb{R}^d$ and a probability measure $\mu$ such that $\supp\mu\subset S$, $\hat{\mu}\in L_{p}$,  $\mathcal{H}_{\alpha}(S)=0$ and $p=\frac{2d}{\alpha}$.
		\end{Th}
		
		The present paper may be considered as an extension of~\cite{Mynotyet}.
		
	\subsection{History}
	
	We will give a brief outline of the history of the problem. See~\cite{Mynotyet} for more historic details.
	The case $d=1$ and $\zeta$ being a measure of Theorem~\ref{nonend} was considered by Salem (see~\cite{Salem}).   
	Beurling obtained a result  that implies Theorem~\ref{nonend} in the case $d=1$ (see~\cite{Beurling}).
	Edgar and Rosenblatt  proved Theorem~\ref{nonend} in the case $d-1\leqslant \alpha$ (see~\cite{EdgRos}).  
	In full generality, Theorem~\ref{nonend} follows from Kahane's work in~\cite{Kah}.

	Denote by $M^+(K)$ is the set of finite non-negative measures on $K$.	
	\begin{Def}
		Define the Fourier dimension of a set $K$:
		\begin{equation}
			\dim_F(K)=\sup\{\alpha\geqslant0\mid\ \exists \mu\in M^+(K) \text{ such that } \mu\neq 0 \text{ and } \hat{\mu}(\xi)=O(|\xi|^{-\frac{\alpha}{2}}) \}.
		\end{equation} 
	\end{Def}

	A compact set $K\subset \mathbb{R}^d$ is called a Salem set if $\dim_H(K)=\dim_F(K)$. Here $\dim_H(K)$ is the Hausdorff dimension of $K$. Every Salem set and Salem measure on it says that Theorem~\ref{nonend} cannot be strengthened to any $p>\frac{2d}{\alpha}$. 	
	Let $K$ be a Salem set of dimension $\alpha_0$, for $\frac{2d}{p}<\alpha_0<\alpha$. Then there exists measure $\mu\in M^+(K)$, such that $\hat{\mu}(\xi)=O(|\xi|^{\frac{-\alpha}{2}})$. Therefore,
	\begin{equation}
		\int\limits_{\mathbb{R}^d}|\hat{\mu}(\xi)|^pd\xi
		\lesssim
		\int\limits_{\mathbb{R}^d}\frac{1}{1+|\xi|^{\frac{\alpha p}{2}}}d\xi
		<
		\infty.
	\end{equation}
	Here and in what follows $A\lesssim B$ means there exists $C$ such that $A\leqslant CB$ and $C$ is uniform in a certain sense, and $A\asymp B$ means that $A\lesssim B$  and $B\lesssim A$.
	
	First examples of Salem sets in the case $d=1$ were found by Salem, see~\cite{Salem}. For arbitrary $d$, Salem sets were constructed by Kahane (see~\cite{Kah1985}).

	In the limit case $\alpha=2d/p$ the UP does not hold (see~\cite[Theorem~1.8]{Mynotyet}). But if one makes additional regularity assumptions on $S$, the UP holds true at the endpoint $\alpha=2d/p$, and here we recollect several results.


	\begin{Th}[Rosenblatt \cite{Ros}]\label{Rosenblatt}
		Let $S\subset \mathbb{R}^d$ be a $(d-1)$-dimensional smooth surface. Let $\zeta$ be a distribution such that $\supp\zeta\subset S$ and $\hat{\zeta}\in  L_{\frac{2d}{d-1}}(\mathbb{R}^d)$. Then $\zeta=0$.
	\end{Th}

	\begin{Th}[Agranovskiy, Narayanan \cite{ArgNar}]\label{AN}
		Let $S\subset \mathbb{R}^d$ be a $C^1$-smooth surface of dimension $k$. Let $\zeta$ be a distribution such that $\supp\zeta\subset S$ and $\hat{\zeta}\in  L_{\frac{2d}{k}}(\mathbb{R}^d)$. Then $\zeta=0$.
	\end{Th}
	Let $\mathcal{P}_{\alpha}$ be the $\alpha$-packing measure (we will recall its definition in Section~\ref{Pre}, see~\cite[Chapter~5]{Mattila} for more information about packing measures).
	\begin{Th}[Raani \cite{Raani}]\label{Raani} 
		Let $S\subset \mathbb{R}^d$ be a compact set such that $\mathcal{P}_{\alpha}(S)<\infty$. Let $\zeta$ be a distribution such that $\supp\zeta\subset S$ and $\hat{\zeta}\in L_{\frac{2d}{\alpha}}(\mathbb{R}^d)$. Then $\zeta=0$.
	\end{Th}
	Theorem~\ref{Raani} is a generalization of Theorems~\ref{Rosenblatt} and~\ref{AN}.	
		
	Netrusov (see~\cite{Net}) studied capacities on Besov spaces, his results imply the following theorem (see~Remark~\ref{vN}).

	\begin{Th}\label{Netcap} 
		Let $2<p<\infty$ and $1\leqslant q<\infty$. Let $S\subset\mathbb{R}^d$ be a compact set such that $\mathcal{H}_{\frac{2d}{p},\frac{q'}{2}}(S)=0$. Let $\zeta$ be a distribution such that $\hat{\zeta}\in L_{p,q}(\mathbb{R}^d)$ and $\supp\zeta\subset{S}$. Then $\zeta=0$.
	\end{Th}
	
	Here $L_{p,q}$ is the Lorentz space and $\mathcal{H}_{\alpha,q}$ is the Netrusov--Hausdorff capacity (we will recall their definitions in the next subsection).

	\subsection{Results}

	We remind the reader the definition of the Lorentz space.
	\begin{Def}
		Let $X$ be a set and $\mu$ be a measure on it.
		Let $p,q\in(0,\infty)$. Then Lorentz spaces $L_{p,q}(X,\mu)$ and $L_{p,\infty}(X,\mu)$ are defined by the quasi-norms
		\begin{equation}
			\|f\|^q_{L_{p,q}}=q\int_{0}^{\infty}(m_f(t)t^p)^{\frac{q}{p}}\frac{dt}{t}=\frac{q}{p}\int_0^{\infty}f^*(t)^qt^{\frac{q}{p}}\frac{dt}{t},
		\end{equation} 
		
		\begin{equation}
			\|f\|_{L_{p,\infty}}=\sup\{m_f(t)^{\frac{1}{p}}t\}
			=
			\sup\{f^*(t)t^{\frac{1}{p}}\}.
		\end{equation} 
		Here $m_f(t)=\mu(\{x|\ |f(x)|\geqslant t\})$ is the distribution function of $f$ and $f^*(t)=\inf\{a|\ m_f(a)\leqslant t\}$ is the decreasing rearrangement of $f$.
		The Lorentz sequence space $\ell_{p, q}$ is a special case of Lorentz spaces for $X=\mathbb{N}$ and $\mu$ being the counting measure.
	\end{Def}
	
	The $L_{p,p}$-seminorm is equal to the standard $L_p$-norm. For $q_1<q_2$ there is a continuous embedding $L_{p,q_1}\hookrightarrow L_{p,q_2}$.	
	See Chapter 1.4 in the book~\cite{Grafakos2014} for basic information about Lorentz spaces.
	
	To formulate our result, we also recall the definitions of the Netrusov--Hausdorff capacities introduced in~\cite{Net}.
	\begin{Def}
		Let $\delta>0$. Let $\Delta_k=\{x\in\mathbb{R}|\ 2^{-k-1}\leqslant x<2^{-k}\}$. The Netrusov--Hausdorff capacity $\mathcal{H}_{\alpha,q}$ is defined by the following formula
		\begin{equation}			
			\mathcal{H}_{\alpha,q}(F,\delta)
			=
			\inf\Big\{\sum\limits_{k}\Big(\sum\limits_{\diam B_j\in\Delta_k}(\diam B_j)^\alpha\Big)^q\Big|\ F\subset\cup B_j,\ \diam B_j <\delta\Big\},\ q<\infty,
		\end{equation}
		\begin{equation}			
			\mathcal{H}_{\alpha,\infty}(F,\delta)
			=
			\inf\Big\{\sup\limits_{k}\sum\limits_{\diam B_j\in\Delta_k}(\diam B_j)^\alpha\Big|\ F\subset\cup B_j,\ \diam B_j<\delta\Big\},
		\end{equation}
		\begin{equation}			
			\mathcal{H}_{\alpha,q}(F)
			=
			\lim\limits_{\delta\rightarrow 0}\mathcal{H}_{\alpha,q}(F,\delta).
		\end{equation}
	\end{Def}
	The capacity $\mathcal{H}_{\alpha,1}$ is equal to the standard Hausdorff measure $\mathcal{H}_\alpha$. 
	See Proposition~\ref{HLp} below for the basic properties of $\mathcal{H}_{\alpha,q}$. 
	The quantity
	\begin{equation}
		\Big(\sum\limits_{k}\Big(\sum\limits_{|a_j|\in\Delta_k}|a_j|^\alpha\Big)^q\Big)^{\frac{1}{q\alpha}}
	\end{equation}
	is equivalent to the value of the quasi-norm of the sequence $a=\{a_j\}_j$ in the Lorentz-sequence space $\ell_{\alpha,\alpha q}$  (see Lemma~\ref{Lornor}). So, the Netrusov--Hausdorff set functions satisfy
	\begin{equation}\label{NHq}		
		\mathcal{H}_{\alpha,q}(F,\delta)\asymp\underset{\diam B_j<\delta}{\underset{F\subset\cup B_j}{\inf}}\|\{\diam B_j\}_{j=1}^{\infty}\|^{\alpha q}_{\ell_{\alpha,\alpha q}},
	\end{equation}
	\begin{equation}\label{NHi}			
		\mathcal{H}_{\alpha,\infty}(F,\delta)\asymp\underset{\diam B_j<\delta}{\underset{F\subset\cup B_j}{\inf}}\|\{\diam B_j\}_{j=1}^{\infty}\|^{\alpha}_{\ell_{\alpha,\infty}}.
	\end{equation}

	Thus, the  Netrusov--Hausdorff capacities can be perceived as something in between Hausdorff measures and Lorentz spaces.

	  Let $q'$ be the dual exponent to $q$, i.e $\frac{1}{q}+\frac{1}{q'}=1$. Here are the main results of the paper.

	\begin{Th}\label{nonendmy} 
		Let $2<p<\infty$ and $1\leqslant q<\infty$. Let $S\subset\mathbb{R}^d$ be a compact set such that $S$ is $\sigma$-finite relatively to $\mathcal{H}_{\frac{2d}{p},\frac{q'}{2}}$. Let $\zeta$ be a distribution such that $\hat{\zeta}\in L_{p,q}(\mathbb{R}^d)$ and $\supp\zeta\subset{S}$. Then $\zeta=0$.
	\end{Th}
	
	\begin{Th}\label{merm}
		Let $2<p<\infty$ and $1\leqslant q<\infty$. Let $\mu$ be  a measure of bounded variation such that $\hat{\mu}\in L_{p,q}(\mathbb{R}^d)$. Then for every set $A\subset\mathbb{R}^d$ such that $\mathcal{H}_{\frac{2d}{p},\frac{q'}{2}}(A)<\infty$, we have $\mu(A)=0$.
	\end{Th}
	
	\begin{Th}\label{resL}
		Let $2<p<\infty$, $1< q<\infty$ and $\beta>\frac{q'}{2}$. There exists a compact set $S\subset \mathbb{R}^d$ and a probability measure $\mu$ such that $\supp\mu\subset S$, $\hat{\mu}\in L_{p,q}(\mathbb{R}^d)$ and $\mathcal{H}_{\frac{2d}{p},\beta}(S)=0$.
	\end{Th}
	
	Theorem~\ref{nonendmy} is a generalization of Theorems~\ref{Raani} and~\ref{Netcap}. Theorem~\ref{resL} says that Theorem~\ref{nonendmy} is sharp. Also Theorem~\ref{merm} says that for measures of bounded variation the UP can be formulated in a stronger form than for a general distribution. 

	%

	In Section~\ref{Pre}, we provide some preliminary information.
	In Section~\ref{caps}, we study general capacities and obtain basic properties of the Netrusov--Hausdorff capacities. 
	In Section~\ref{S2}, we prove Theorems~\ref{nonendmy} and~\ref{merm} by using results from section~\ref{caps}.
	In Section~\ref{S3}, we prove Theorem~\ref{resL}.
	In Section~\ref{tech}, we give all technical proofs. 
	In Section~\ref{genops}, we provide a generalizations and formulate an open question.

	{\bf Acknowledgment.}
	I am grateful to my scientific adviser D. M. Stolyarov for statement of the problem and attention to my work.

	\section{Preliminaries and notation}\label{Pre}
	We will use the Ky Fan minimax Theorem (see~\cite[Theorem 2]{Fan}) and provide its formulation for the convenience of the reader. 
	\begin{Def}
		Let $f$ be a real-valued function defined on the product $X\times Y$ of two arbitrary sets $X$ and $Y$. A function $f$ is said to be convex on $X$, if for any two elements $x_1,x_2\in X$ and two numbers $\alpha_1\geqslant0$, $\alpha_2\geqslant0$ with $\alpha_1+\alpha_2=1$, there exists an element $x_0\in X$ such that $f(x_0,y)\leqslant \alpha_1f(x_1,y)+\alpha_2f(x_2,y)$ for all $y\in Y$.
		Similarly $f$ is said to be concave on $Y$, if for any two elements $y_1,y_2\in Y$ and two numbers $\alpha_1\geqslant0$, $\alpha_2\geqslant0$ with $\alpha_1+\alpha_2=1$, there exists an element $y_0\in Y$ such that $f(x,y_0)\geqslant \alpha_1f(x,y_1)+\alpha_2f(x,y_2)$ for all $x\in X$.
	\end{Def}
	Here the choice of $x_0$ depends on $x_1$, $x_2$, $\alpha_1$, $\alpha_2$ and $f$ only, but does not depend on $y$. Similarly, $y_0$ does not depend on $x$. Therefore, the definition above depends on the product structure of $X\times Y$.
	\begin{Th}[Ky Fan minimax Theorem]\label{KyF}
		Let $X$ be a compact Hausdorff space and let $Y$ be an arbitrary set. Let $f$ be a real-valued function on $X\times Y$ such that for every $y\in Y$, $f(\cdot,y)$ is lower semi-continuous on $X$. If $f$ is convex on $X$ and concave on $Y$, then
		\begin{equation}
			\min\limits_{\raisebox{7px}{} x\in X}\sup\limits_{y\in Y} f(x,y)
			=
			\sup\limits_{y\in Y}\min\limits_{\raisebox{7px}{}x\in X} f(x,y).
		\end{equation}
	\end{Th}
	Here we do not need any topological structure on $Y$.
	
	We need the following theorem for Lorentz spaces (see~\cite[Chapter IV, Theorem~3.13]{SteinWeiss})
	\begin{Th}\label{SW}
		Let $X$ be a set and $\mu$ be a measure on it. Suppose $T$ is a linear operator which maps the finite linear combinations of characteristic functions $\chi_{E}$ of set $E\subset X$ of finite measure into a Banach space $B$. If
		\begin{equation}
			\|T\chi_E\|_B
			\lesssim 
			\|\chi_E\|_{L_{p,1}}
			=
			\|\chi_E\|_{L_p},
		\end{equation}
		then
		\begin{equation}
			\|Tf\|_B
			\lesssim
			\|f\|_{L_{p,1}}
		\end{equation}
		for all $f$ in domain of $T$.
	\end{Th}
	\begin{Rem}
		In the fact in~\cite{SteinWeiss} this theorem is formulated for Banach spaces with an order preserving norm, but the proof never uses the additional structure on the Banach space $B$. 
	\end{Rem}
	We will use real interpolation between quasi-Banach spaces (see~\cite{BerghLofstrom} for basics of interpolation theory). 
	
	\begin{Def}
		Let $A_0$ and $A_1$ be two topological vector spaces. Then we will say that $A_0$ and $A_1$ are compatible if there is a Hausdorff topological vector space $\mathcal{A}$  such that $A_0$ and $A_1$ are subspaces of $\mathcal{A}$.
	\end{Def}
	\begin{Def}
		Let $A_0$ and $A_1$ be two compatible quasi-Banach spaces. Define the $K$-functional
		\begin{equation}
			K(\,\cdot\,,\,\cdot\,,A_0,A_1)\colon\mathbb{R}_+\times(A_0+A_1)\rightarrow \mathbb{R}_+
		\end{equation} 
		by the formula
		\begin{equation}
			K(t,a,A_0,A_1)=\inf\{\|a_0\|_{A_0}+t\|a_1\|_{A_1}| a_0\in A_0, a_1\in A_1\text{ such that }a=a_0+a_1\}.
		\end{equation}
	\end{Def}
	\begin{Def}
		Let $A_0$ and $A_1$ be two compatible quasi-Banach spaces. Let $0<\theta<1$ and $q>0$. The real interpolation space $(A_0,A_1)_{\theta,q}$ is defined by the quasi-norm
		\begin{equation}
		\|a\|_{(A_0,A_1)_{\theta,q}}
		=
		\Big(\int\limits_{0}^{\infty}(t^{-\theta}K(t,a,A_0,A_1))^q\frac{dt}{t}\Big)^{\frac{1}{q}}.
		\end{equation}
	\end{Def}
	In the next proposition, we collect the properties of real interpolation spaces we will use in our paper.
	\begin{St}\label{inter}
		Let $A_0$ and $A_1$ be two compatible quasi-Banach spaces and 	let $X_0$ and $X_1$ also be two compatible quasi-Banach spaces. 
		\begin{enumerate}[1.]
			\item\label{I1}~\cite[Theorem~3.11.4]{BerghLofstrom}. $\|a\|_{(A_0,A_1)_{\theta,q}}\lesssim\|a\|_{A_0}^{1-\theta}\|a\|_{A_1}^{\theta}$.
			\item\label{I2}~\cite[Theorem~3.11.2]{BerghLofstrom}.
			Let the linear operator $T\colon A_0+A_1\rightarrow X_0+X_1$ be continuous from $A_0$ to $X_0$ and from $A_1$ to $X_1$. Then $T$ is continuous from $(A_0,A_1)_{\theta,q}$ to $(X_0,X_1)_{\theta,q}$.
			\item\label{I3}~\cite[Theorem~3.11.5]{BerghLofstrom}.
			Let $\theta=(1-\eta)\theta_0+\eta\theta_1$, then 
			\begin{equation}
				\big((A_0,A_1)_{\theta_0,q_0},(A_0,A_1)_{\theta_1,q_1}\big)_{\eta,q}=(A_0,A_1)_{\theta,q}.
			\end{equation}
			\item\label{I4}~\cite[Theorem~5.3.1]{BerghLofstrom}.
			Let $\frac{1}{p}=\frac{1-\theta}{p_0}+\frac{\theta}{p_1}$, then $(L_{p_0,q_0},L_{p_1,q_1})_{\theta,q}=L_{p,q}$.
		\end{enumerate}
	\end{St}
	
	\begin{Def}
		A Banach space $X$ is called a Banach space of distributions if $S(\mathbb{R}^d)\subset X\subset S'(\mathbb{R}^d)$ and these embeddings are continuous. Furthermore, $X$ is a standard Banach space of distributions if $S(\mathbb{R}^d)$ is dense in $X$.
	\end{Def}
	Here $S(\mathbb{R}^d)$ is the Schwartz class and $S'(\mathbb{R}^d)$ is the class of tempered distributions.
	
	\begin{Rem}\label{dualsp}
		If $X$ is a standard Banach space of distributions, then $S(\mathbb{R}^d)\hookrightarrow X'\hookrightarrow S'(\mathbb{R}^d)$ continuously, i.e. $X'$ is a Banach space of distributions. We mean that $X$ and $X'$ are Banach spaces and subspaces of $S'(\mathbb{R}^d)$. The action of $x'\in X'$ as a linear functional on the space $X$ is consistent with the action of $x'\in S'(\mathbb{R}^d)$ as a linear functional on the space $S(\mathbb{R}^d)$.
	\end{Rem}

	We will also need Besov and Sobolev spaces. Let $B_r(x)$ be the open ball with center $x$ and radius $r$. Fix a function $\varphi\in C_0^\infty(\mathbb{R}^d)$ such that $\supp\varphi\subset B_1(0)$ and $\varphi(x)=1$ if $|x|\leqslant \frac{1}{2}$. Let $\varphi_n(x)=\varphi(\frac{x}{2^n})-\varphi(\frac{x}{2^{n-1}})$, $\psi=\hat{\varphi}$, $\psi_n=\hat{\varphi}_n$.
	\begin{Def}
		Let $s\in\mathbb{R}$, $1\leqslant p\leqslant\infty$ and $0<q<\infty$. Then we define Besov spaces by the formulas
		\begin{equation}
			B_{p,q}^s(\mathbb{R}^d)
			=
			\Big\{f\Big|\  f\in S'(\mathbb{R}^d)\text{ and } \|\psi*f\|_{L_p}+\big(\sum\limits_{n=0}^{\infty}(2^{ns}\|\psi_n*f\|_{L_p})^q\big)^{\frac{1}{q}}<\infty\Big\}
			,
		\end{equation}
		\begin{equation}
			B_{p,\infty}^s(\mathbb{R}^d)
			=
			\Big\{f\Big|\  f\in S'(\mathbb{R}^d)\text{ and } \|\psi*f\|_{L_p}+\sup\limits_{n\geqslant0}\big(2^{ns}\|\psi_n*f\|_{L_p})<\infty\Big\},
		\end{equation}
		\begin{equation}\label{close}
			\overset{\!\circ}{B^s}_{\!\!\!\!\!\; p,\infty}(\mathbb{R}^d)
			=
			clos_{B_{p,\infty}^s}(S(\mathbb{R}^d)).
		\end{equation}
	\end{Def}
	The formula~\eqref{close} means that $\overset{\!\circ}{B^s}_{\!\!\!\!\!\; p,\infty}(\mathbb{R}^d)$ is the minimal closed subset of $B_{p,\infty}^s$ that contains $S(\mathbb{R}^d)$.
	More properties of Besov spaces may be found in the books~\cite{BesIlNik},~\cite{Peet} and~\cite{Triebel}. 
	
	We use the notation $(1-\Delta)^{\beta/2}$ for 
	\begin{equation}
		((1-\Delta)^{\beta/2}f)^{\widehat{}}(\xi)=\hat{f}(\xi)(1+|2\pi \xi|^2)^{\beta/2}.
	\end{equation}
	\begin{Def}
		Let $s\in\mathbb{R}$ and $1\leqslant p\leqslant\infty$. Then we define the potential Sobolev spaces by the formula
		\begin{equation}
			W^{s}_p(\mathbb{R}^d)=\{f| \  (1-\Delta)^{\frac{s}{2}}f\in L_p(\mathbb{R}^d)\}.
		\end{equation}		
	\end{Def}

	In the next proposition, we collect the properties of Besov spaces we need in our work.
	\begin{St}\label{BesPr}
		Besov spaces satisfy the following statements.
		\begin{enumerate}[1.]
			\item\label{B1}~\cite[Theorem 12, p.~74]{Peet} and~\cite[paragraph~2.11.2 Remark~2, p.~180]{Triebel}.
			Let $1\leqslant p<\infty$ and $1\leqslant q<\infty$. Then, 
			\begin{equation}\label{B1.1}
				(B^{s}_{p,q})'=B^{-s}_{p',q'},
			\end{equation}
			
			\begin{equation}\label{B1.2}
				\left(\overset{\!\circ}{B^s}_{\!\!\!\!\!\; p,\infty}\right)'=B^{-s}_{p',1}.
			\end{equation}
			\item\label{B2}~\cite[Theorem 1, p.~79]{Peet}.
			$B^s_{2,2}(\mathbb{R}^d)=W^s_2(\mathbb{R}^d)$.
			
			\item\label{B3}~\cite[Theorem 7, p.~64]{Peet}.
			Let $0<\theta<1$, $q>0$, $1< p<\infty$ and $s=(1-\theta)s_0+\theta s_1$. Then,
			\begin{equation}\label{intS}
				(W_p^{s_0},W_p^{s_1})_{\theta,q}=B^s_{p,q}.
			\end{equation} 
			\item\label{B4}This item follows from \cite[Theorem 4, p.~199]{Peet} and the previous items.
			If $2<p<\infty$ and $q>0$, then 
			\begin{equation}
				\mathcal{F}\colon L_{p,q}(\mathbb{R}^d)\rightarrow B^{\frac{d}{p}-\frac{d}{2}}_{2,q}(\mathbb{R}^d).
			\end{equation}
			Here $\mathcal{F}$ is the Fourier transform operator.
		\end{enumerate}
	\end{St}

		We give four definitions of capacities. Let $X$ be a real Banach space of distributions and let $A\subset\mathbb{R}^d$ be a compact set. We introduce the closed convex sets
		\begin{equation}
			X_A^+=clos_X\{g\in L_{1\ \!\! loc}\cap X| \ \ g\geqslant1 \text{ in some neighborhood of } A\},
		\end{equation}
		\begin{equation}
			X_A=clos_X\{g\in L_{1\ \!\! loc}\cap X| \ \ g=1 \text{ in some neighborhood of } A\}.
		\end{equation}
		For complex Banach spaces of distributions the definition of $X_A$ is the same, and 
		\begin{equation}
			X_A^+=clos_X\{g\in L_{1\ \!\! loc}\cap X| \ \ \Re(g)\geqslant1 \text{ in some neighborhood of } A\}.
		\end{equation}
		\begin{Rem}\label{nonemp}
			Let $X$ be a Banach space such that $S(\mathbb{R}^d)\subset X$. Then $X_A^+\neq\emptyset$ and $X_A\neq\emptyset$.
		\end{Rem}
		\begin{Def}
			Let $X$ be a Banach space of distributions, and let $A\subset\mathbb{R}^d$ be a compact set. Define
			\begin{equation}
				cap(A,X)=\inf\{\|g\|_X| \ g\in X^+_A\},
			\end{equation}
			\begin{equation}
				Cap(A,X)=\inf\{\|g\|_X| \ g\in X_A\},
			\end{equation}
			\begin{equation}
				cap^*(A,X)=\sup\{\|\mu\|^{-1}_X| \ \ \mu \text{ is a positive measure such that } \supp\mu\subset A\text{ and }\mu(A)=1\},
			\end{equation}
			\begin{equation}
				Cap^*(A,X)=\sup\{\|\zeta\|^{-1}_X| \ \ \zeta \text{ is a distribution such that } \supp\zeta\subset A\text{ and }<\zeta,1>=1\}.
			\end{equation}
		\end{Def}
		
		We also need the following lemma (see~\cite[Proposition 1.2]{Net}).
		\begin{Le}\label{Netrus}
			Let $1\leqslant p<\infty$, $0<q<\infty$ and $s>0$. Then, 
			\begin{equation}
				cap(A,B^{s}_{p,q})\asymp Cap(A, B^{s}_{p,q}).
			\end{equation}
		\end{Le}

		\begin{Def}\label{regyl}
			A continuous non-decreasing function~$g\colon\mathbb{R}_+\rightarrow \mathbb{R}_+$ is called regular provided~$g(0)=0$, and for some fixed~$c > 1$ we have~$g(x)\asymp g(cx)$.
		\end{Def}
		
		We recall the definitions of $f$-Hausdorff measures, see~\cite[Chapter~4, p~59]{Mattila} for more information about $f$-Hausdorff measures.
		\begin{Def}
			The $f$-Hausdorff measure is defined by the formula  
			\begin{equation}
				\Lambda_f(A)
				=
				\lim\limits_{\delta\rightarrow 0} \inf\Big\{\sum\limits_{j} f(\diam B_j)\Big|\ A\subset\cup B_j,\ \diam B_j<\delta\Big\}.
			\end{equation}
		\end{Def}
		Note that if $f(t)=t^{\alpha}$, then $\Lambda_f=\mathcal{H}_{\alpha}$. The	$f$-Hausdorff measures provide a way to expand the scale of the classical Hausdorff measures.	Items~\ref{I8} and~\ref{I9} of Proposition~\ref{HLp} show how this expansion is related to Netrusov--Hausdorff capacities.
		We recall definitions of the $\alpha$-packing measures, see~\cite[Chapter~5, p~81]{Mattila} for more information about packing measures. A disjoint collection of open balls with centers in $A$ is called an $A$ centered ball packing (CBP).
		\begin{Def}
			The $\alpha$-packing pre-measure $\widetilde{\mathcal{P}}_\alpha$ is defined by the formula
			\begin{equation}
				\widetilde{\mathcal{P}}_\alpha(F)
				=
				\lim\limits_{\delta\rightarrow 0} \sup\Big\{\sum\limits_j (\diam B_j)^{\alpha}\Big|\ \{B_j\}\text{ is } F \text{ CBP},\ \diam B_j<\delta\Big\}.
			\end{equation}
		\end{Def}
		\begin{Def}
			The $\alpha$-packing measure $\mathcal{P}_\alpha$ is defined by the formula
			\begin{equation}
				\mathcal{P}_\alpha(F)
				=
				\underset{F\subset\cup B_j}{\inf}\sum\limits_{j}\widetilde{\mathcal{P}}_\alpha(B_j).
			\end{equation}
		\end{Def}
		\begin{Rem}\label{yp}
			Let $N_\varepsilon(A)$ be the number of points in a maximal $\varepsilon$-separated subset of set $A$. Then $N_\varepsilon(A)\lesssim \widetilde{\mathcal{P}}_\alpha(A)\varepsilon^{-\alpha}$.
		\end{Rem}
		We also recall the definition of the lower Hausdorff dimension of a measure, see~\cite[Chapter 10]{Falconer} for more information about Hausdorff dimension of a measure.
		\begin{Def}
			Let $\mu$ be a Borel measure in $\mathbb{R}^d$. Then the lower Hausdorff dimension of $\mu$ is defined by the formula
			\begin{equation}
				\underline{\dim}_{\mathcal{H}}(\mu)
				=
				\inf\{\alpha|\ \text{There exists a borel } B\subset\mathbb{R}^d\text{ such that }\dim_{\mathcal{H}}(B)=\alpha\text{ and }\mu(B)\neq0\}.
			\end{equation}
		\end{Def}
	\section{About capacities}\label{caps}
	
	The Ky Fan Minimax Theorem implies the equivalence of the dual definitions of capacities. See~\cite[Section 2.5]{AdHe} for duality between definitions of capacities and in particular, this deduction. We repeat similar reasoning in a general situation.
	\begin{St}\label{dual}
		Let $X$ be a real standard Banach space of distributions. Let $A\subset\mathbb{R}^d$ be a compact set. Then
		\begin{equation}\label{cap}
			cap(A,X)=cap^*(A,X'),
		\end{equation}
		\begin{equation}\label{Cap}
			Cap(A,X)=Cap^*(A,X').
		\end{equation}
	\end{St}
	\begin{proof}[Proof of Proposition~\ref{dual}.] 
		Let us prove~\eqref{cap}, the second case~\eqref{Cap} is similar. We employ the Hahn--Banach Theorem:
		\begin{equation}
			\inf\{\|g\|_X| \ \ g\in X_A^+\}=\inf\limits_{g\in X^+_A}\sup\limits_{\|\mu\|_{X'}\leqslant 1}<g,\mu>.
		\end{equation}
		Note that the sets $X_A^+$ and $\{\mu\in X'|\ \|\mu\|_{X'}\leqslant 1\}$ are convex, the latter set is compact with respect to the weak* topology, and the function $(g,\mu)\mapsto <g,\mu>$ is bilinear and continuous. Thus, by Theorem~\ref{KyF}, 
		\begin{equation}
			\inf\limits_{g\in X^+_A}\sup\limits_{\|\mu\|_{X'}\leqslant 1}<g,\mu>=\sup\limits_{\|\mu\|_{X'}\leqslant 1}	\inf\limits_{g\in X^+_A}<g,\mu>.
		\end{equation}
		Let us now fix $\mu$ for a while. Let $g=g_0+\lambda f$ where $g_0\in X_A^+$, $\lambda\in\mathbb{R}$ and $\supp f\cap A=\varnothing$. Therefore, either $\inf_{X_A^+}<g,\mu>=-\infty$ or $<f,\mu>=0$ for any such $f$. This means that $\mu$ is supported on $A$, provided $\inf_{X^+_A}<g,\mu>$ is finite. Now let us consider $f\in C_0^\infty$ that are non-negative and $\lambda\geqslant0$. For each such $f$, we either have $\inf_{X_A^+}<g,\mu>=-\infty$ or $<f,\mu>\geqslant0$. This means that $\mu$ is a positive distribution, therefore, by the Schwartz Theorem, $\mu$ is a non-negative measure. All these leads to the formula
		\begin{equation}
			cap(A,X)=\sup\{\inf\limits_{X^+_A}<g,\mu>|\ \|\mu\|_{X'}\leqslant1 \text{ and }\mu \text{ is a non negative measure supported on }A\}.
		\end{equation}
		Now, if $\mu\in X'$ is a positive measure supported on $A$, then
		\begin{equation}\label{poia}
			\inf\limits_{g\in X^+_A}<g,\mu>
			=
			\inf\{<g,\mu>|\ g\in L_{1\ \!\! loc}\cup X \text{ and } g\geqslant1 \text{ in some neighborhood of } A \}
			=
			\mu(A).
		\end{equation}
		The last equality in~\eqref{poia} follows from Remark~\ref{nonemp}.
		
		Thus,
		\begin{equation}
			\begin{aligned}
				cap(A,X)&=\sup\{\mu(A)|\ \|\mu\|_{X'}\leqslant1 \text{ and }\mu \text{ is a positive measure supported on }A\}\\
				&=
				\sup\{\|\mu\|_{X'}^{-1}|\ \mu(A)=1 \text{ and }\mu \text{ is a positive measure supported on }A\}\\
				&=
				cap^*(A,X').
			\end{aligned}
		\end{equation}
	\end{proof}
	
	\begin{Rem}
		For Besov spaces of complex-valued functions,  we have the relation
		\begin{equation}
			\|g\|_{B^s_{p,q}}\asymp\|\Re(g)\|_{B^s_{p,q}}+\|\Im(g)\|_{B^s_{p,q}}.
		\end{equation}
		So, the capacities in real Besov spaces are equivalent to the capacities in complex Besov spaces.
	\end{Rem}
	
	\begin{Le}\label{capz}
		Let $s>0$ and $1<p<\infty$, then
		\begin{equation}\label{zerocap}
			cap(\,\cdot\,,B^s_{p,\infty})
			\asymp
			cap(\,\cdot\,,\overset{\!\circ}{B^s}_{\!\!\!\!\!\; p,\infty}),
		\end{equation}
		\begin{equation}
			Cap(\,\cdot\,,B^s_{p,\infty})
			\asymp
			Cap(\,\cdot\,,\overset{\!\circ}{B^s}_{\!\!\!\!\!\; p,\infty}).
		\end{equation}
	\end{Le}
	\begin{proof}
		Let us prove~\eqref{zerocap}, the second case is similar. The embedding $\overset{\!\circ}{B^s}_{\!\!\!\!\!\; p,\infty}\hookrightarrow B^s_{p,\infty}$ yields
		\begin{equation}\label{ocenkacap}
			cap(\,\cdot\,,B^s_{p,\infty})\leqslant cap(\,\cdot\,,\overset{\!\circ}{B^s}_{\!\!\!\!\!\; p,\infty}).
		\end{equation}
		Let $g\in L_{1\ \!\! loc}\cup B^s_{p,\infty}$ be such that  $g\geqslant1$ in some neighborhood of  $A$.
		To prove the inequality reverse to~\eqref{ocenkacap}, we need a function $\varphi\in C_0^{\infty}(\mathbb{R}^d)$ such that $\varphi=1$ in the unit ball $B_1(0)$. Let $\varPhi_R(x)=\varphi(\frac{x}{R})$ and $\varphi_\varepsilon(x)=\frac{1}{C\varepsilon^d}\varphi(\frac{x}{\varepsilon})$ (here $C=\int_{\mathbb{R^d}} \varphi(x)dx$). Let $g\in B^{s}_{p,\infty}$ and $g\geqslant 1$ (or $g=1$) in some neighborhood of $S$. Then, for large $R$ and small $\varepsilon$ we have
		\begin{equation}
			\|(g\varPhi_R)*\varphi_\varepsilon\|_{B^s_{p,\infty}}\lesssim\|g\|_{B^s_{p,\infty}},
		\end{equation}
		moreover, $(g\varPhi_R)*\varphi_\varepsilon\in C_0^\infty(\mathbb{R}^d)$, and $(g\varPhi_R)*\varphi_\varepsilon\geqslant 1$ in some neighborhood of $S$. Therefore
		\begin{equation}
			cap(A,\overset{\!\circ}{B^s}_{\!\!\!\!\!\; p,\infty})\lesssim \|g\|_{B^s_{p,\infty}}.
		\end{equation}
		Passing to the infimum over all suitable $g$, we get 
		 \begin{equation}
		 	cap(A,\overset{\!\circ}{B^s}_{\!\!\!\!\!\; p,\infty})\lesssim cap(A,B^s_{p,\infty}).
		 \end{equation}
	\end{proof}

	\begin{Cor}\label{eqvi}
		Let $1<p<\infty$, $1\leqslant q\leqslant\infty$ and $0<s<\infty$. Let $S\subset\mathbb{R}^d$ be a compact set. Let there exist a distribution $\zeta$ such that $\zeta\neq0$, $\supp\zeta\subset S$ and $\zeta\in B^{-s}_{p,q}$. Then there exists a probability measure $\mu$ such that $\supp\mu\subset S$ and $\mu\in B^{-s}_{p,q}$.
	\end{Cor}

	\begin{proof} Case $q>1$.
		
		The existence of a distribution $\zeta$ described above is equivalent to the inequality $Cap^*(S,B^{-s}_{p,q})>0$. By Proposition~\ref{dual}, this may be reduced to $Cap(S,B^s_{p',q'})>0$. So, we have the inequality $cap(S,B^s_{p',q'})>0$ (Lemma~\ref{Netrus}), use Proposition~\ref{dual} again and obtain the inequality $cap^*(S,B^{-s}_{p,q})>0$. This inequality is equivalent to the existence of the measure $\mu$.

		The case $q=1$ requires a slight modification of the proof.		
		We need the duality formula~\eqref{B1.2} and Lemma~\ref{capz}.
		The rest of the proof is the same as in the case $q>1$.
		
	\end{proof}

	The following proposition collects basic properties of $\mathcal{H}_{\alpha,q}$. The proof is postponed till Section~\ref{tech}.
	\begin{St}\label{HLp}
		The Netrusov--Hausdorff set function satisfies the following properties.
		\begin{enumerate}[1.]
			\item $\mathcal{H}_\alpha=\mathcal{H}_{\alpha,1}$.
			
			\item $\mathcal{H}_{\alpha,q}$ is an outer measure.
			
			\item If $1\leqslant q<\infty$, then $\mathcal{H}_{\alpha,q}$ is a Borel measure.

			\item Let $S$ be a $\sigma$-finite set relative to the $\alpha$-packing measure. Then, for any $q>0$, $S$ is $\sigma$-finite relative to  $\mathcal{H}_{\alpha,q}$.

			\item Let $0<q_1<q_2$. Then $\mathcal{H}_{\alpha,q_2}$ is absolutely continuous with respect to $\mathcal{H}_{\alpha,q_1}$.

			\item Let $1\leqslant q_1<q_2$. Then we have the following statements:
			
			If $\mathcal{H}_{\alpha,q_1}(S)<\infty$, then   $\mathcal{H}_{\alpha,q_2}(S)=0$;
			
			if $\mathcal{H}_{\alpha,q_2}(S)>0$, then   $\mathcal{H}_{\alpha,q_1}(S)=\infty$.
			
			\item Let $\alpha_1<\alpha_2$. Then for any $q_1,q_2$ we have the following statements:
			
			If $\mathcal{H}_{\alpha_1,q_1}(S)<\infty$, then   $\mathcal{H}_{\alpha_2,q_2}(S)=0$;
			
			if $\mathcal{H}_{\alpha_2,q_2}(S)>0$, then   $\mathcal{H}_{\alpha_1,q_1}(S)=\infty$.

			\item\label{I8}
			Let $0<q<1$. Let $f$ be a regular function such that
			\begin{equation}
				\int\limits_{0}^{1}\left(\frac{t^\alpha}{f(t)}\right)^{\frac{q}{1-q}}\frac{dt}{t}<\infty.
			\end{equation}
			Then $\mathcal{H}_{\alpha,q}$ is absolutely continuous with respect to $\Lambda_f$.
			Moreover,

			if $\mathcal{H}_{\alpha,q}(S)>0$, then   $\Lambda_f(S)=\infty$;
			
			if $\Lambda_f(S)<\infty$, then   $\mathcal{H}_{\alpha,q}(S)=0$.

			\item\label{I9} 
			Let $1<q\leqslant\infty$. Let $f$ be a regular function such that
			\begin{equation}
				\int\limits_{0}^{1}\left(\frac{f(t)}{t^\alpha}\right)^{\frac{q}{q-1}}\frac{dt}{t}<\infty.
			\end{equation}
			Then $\Lambda_f$ is absolutely continuous with respect to $\mathcal{H}_{\alpha,q}$.
			Moreover,

			if $\Lambda_f(S)>0$, then   $\mathcal{H}_{\alpha,q}(S)=\infty$;
			
			if $\mathcal{H}_{\alpha,q}(S)<\infty$, then   $\Lambda_f(S)=0$.

		\end{enumerate}
	\end{St}

	We need a version of the Frostman Lemma for $\mathcal{H}_{\alpha,q}$.
	Let $\varphi_{x,r}(y)=\varphi\left(\frac{y-x}{r} \right)$.
	\begin{Le}\label{frost}
		Let~$\varphi$ be a bounded radially symmetric, radially non-increasing compactly supported function such that~$\overline{B}_3(0)=\supp \varphi $. Let~$\mu$ be a signed measure of bounded variation. Assume
		\begin{equation}
			\Big|\sum\limits_{B_{r_j}(x_j)\in \mathfrak{B}}\underset{\mathbb{R}^d}{\int}\varphi_{x_j,r_j}(y) d\mu(y) \Big|\lesssim \Big(\|\{r_j| \ {B_{r_j}(x_j)\in \mathfrak{B}}\} \|^q_{\ell_{\alpha,q}}\Big)^\gamma
		\end{equation}
		for any finite disjoint family of balls~$\mathfrak{B}$ with radii less than $1$. Then, the inequality
		\begin{equation}
			|\mu|(A)\lesssim (\mathcal{H}_{\alpha,\frac{q}{\alpha}}(A))^\gamma
		\end{equation}
		holds true for any Borel set~$A$.
	\end{Le}
	
	This version of the Frostman lemma for the standard Hausdorff measure $\mathcal{H}_\alpha$ was studied in~\cite[Lemma~2.3]{My} (a weaker version of the lemma had firstly appeared in~\cite{StolyarovWojciechowski2014}). An analog of this lemma plays the pivotal role in~\cite{AyStWo2020} (see Lemma 4.2 in that paper), where a simpler discrete analog of the dimension problem for Fourier constrained measures is solved. A similar lemma provides good (better than the ones given by the energy method) dimensional estimates for Riesz products, see~\cite{AyStWo2021} (the proof of Theorem 2.8). The proof of Lemma~\ref{frost} is the same as in~\cite{My}. We give it in Section~\ref{tech} below.
	
	\begin{Cor}\label{dualeq}
		Let~$\varphi$ be a bounded radially symmetric, radially non-increasing function with compact support such that~$\overline{B}_3(0)=\supp \varphi $. Let $X$ be a standard Banach space of distributions. Let $0<\alpha<d$ and $0<q\leqslant\infty$. Assume
		\begin{equation}
			\Big\|\sum\limits_{B_{r_j}(x_j)\in \mathfrak{B}}\varphi_{x_j,r_j} \Big\|_{X'}\lesssim \Big(\|\{r_j| \ {B_{r_j}(x_j)\in \mathfrak{B}}\}\|^q_{\ell_{\alpha,q}}\Big)^\gamma
		\end{equation}
		for any finite disjoint family of balls~$\mathfrak{B}$  with radii of balls less than $1$. Then for any measure of bounded variation $\mu$ and any Borel $A$ such that $\mathcal{H}_{\alpha,\frac{q}{\alpha}}(A)<\infty$ and $\mu\in X$ we have $\mu(A)=0$.
	\end{Cor}
	\begin{Rem}
		In particular, Corollary~\ref{dualeq} says that $\underline{\dim}_\mathcal{H}(\mu)\geqslant\alpha$.
	\end{Rem}
	\begin{proof}[Proof of Corollary~\ref{dualeq}.]
		Let $\mu\in X$ be a measure of bounded variation. Then, 
		\begin{equation}
			\begin{aligned}
				\Big|\sum\limits_{B_{r_j}(x_j)\in \mathfrak{B}}\underset{\mathbb{R}^d}{\int}\varphi_{x_j,r_j}(y) d\mu(y) \Big|
				&=
				\Big|\Big<\sum\limits_{B_{r_j}(x_j)\in \mathfrak{B}}\varphi_{x_j,r_j},\mu\Big>\Big|
				\leqslant
				\Big\|\sum\limits_{B_{r_j}(x_j)\in \mathfrak{B}}\varphi_{x_j,r_j} \Big\|_{X'}\|\mu\|_X\\
				&\lesssim
				\Big(\|\{r_j| \ {B_{r_j}(x_j)\in \mathfrak{B}}\} \|^q_{\ell_{\alpha,q}}\Big)^\gamma\|\mu\|_{X}.
			\end{aligned}
		\end{equation}
		
		Therefore, by Lemma~\ref{frost} we have the inequality $\|\mu\|(A)\lesssim \|\mu\|_X\mathcal{H}_{\alpha,\frac{q}{\alpha}}(A)^\gamma$. Let $g_j\in S(\mathbb{R}^d)$ be a sequence of functions such that $g_j\underset{X}{\rightarrow}\mu$. Then,
		\begin{equation}
			\|\mu\|(A)=\|\mu-g_j\|(A)\lesssim
			\|\mu-g_j\|_X\mathcal{H}_{\alpha,\frac{q}{\alpha}}(A)^\gamma
			\overset{j\rightarrow\infty}{\longrightarrow}0
		\end{equation}
		for any Borel set $A\subset\mathbb{R}^d$ with $\mathcal{H}_{\alpha,\frac{q}{\alpha}}(A)<\infty$. Here we identify a function $g_j$ and the Lebesgue continuous measure with the density $g_j$.
	\end{proof}

	\section{Proof of Theorems~\ref{nonendmy} and~\ref{merm}}\label{S2}

	\begin{Le}\label{strash}
		Let $X_0$ and $X_1$ be compatible Banach spaces. Let $\alpha_0>0$, $\alpha_0>\alpha_1$ and $1\leqslant p<\infty$. Assume
		\begin{equation}
			\Big\|\sum\limits_{B_{r_j}(x_j)\in \mathfrak{B}}\varphi_{x_j,r_j} \Big\|_{X_0}\lesssim \Big(\sum\limits_{B_{r_j}(x_j)\in \mathfrak{B}}r_j^{\alpha_0}\Big)^\frac{1}{p}
		\end{equation}
		and
		\begin{equation}
			\Big\|\sum\limits_{B_{r_j}(x_j)\in \mathfrak{B}}\varphi_{x_j,r_j} \Big\|_{X_1}\lesssim \Big(\sum\limits_{B_{r_j}(x_j)\in \mathfrak{B}}r_j^{\alpha_1}\Big)^\frac{1}{p}
		\end{equation}
		for any finite disjoint family of balls~$\mathfrak{B}$  with radii of balls less than $1$. Let $\alpha_{\theta}=(1-\theta)\alpha_0+\theta\alpha_1$. If $\alpha_\theta>0$ then
		\begin{equation}
			\Big\|\sum\limits_{B_{r_j}(x_j)\in \mathfrak{B}}\varphi_{x_j,r_j} \Big\|_{(X_0,X_1)_{\theta,h}}\lesssim
			\|\{r_j| \ {B_{r_j}(x_j)\in \mathfrak{B}}\}\|^{\frac{\alpha_{\theta}}{p}}_{\ell_{\alpha_\theta,\frac{h\alpha_\theta}{p}}}.
		\end{equation}
	\end{Le}
	Here the condition $\alpha_\theta>0$ is necessary for the space $\ell_{\alpha_\theta,\frac{h\alpha_\theta}{p}}$ to be correctly defined.
	We will prove this lemma in Section~\ref{tech}, since it is justified by standard methods of interpolation theory.
	
	\begin{Le}\label{dd}
		Let~$\varphi\in C^\infty_0(\mathbb{R}^d)$ be a bounded radially symmetric, radially non-increasing compactly supported function such that~$\supp \varphi\subset\overline{B}_3(0)$. Then
		\begin{equation}\label{L_2}
			\Big\|\sum\limits_{B_{r_j}(x_j)\in \mathfrak{B}}\varphi_{x_j,r_j} \Big\|_{L_2(\mathbb{R}^d)}\lesssim
			\Big(\sum\limits_{B_{r_j}(x_j)\in \mathfrak{B}}r_j^{d}\Big)^\frac{1}{2}
		\end{equation}
		and
		\begin{equation}\label{W^d}
			\Big\|\sum\limits_{B_{r_j}(x_j)\in \mathfrak{B}}\varphi_{x_j,r_j} \Big\|_{W^{d}_2(\mathbb{R}^d)}\lesssim \Big(\sum\limits_{B_{r_j}(x_j)\in \mathfrak{B}}r_j^{-d}\Big)^\frac{1}{2}
		\end{equation}
		for any disjoint family of balls~$\mathfrak{B}$  with radii of balls less than $1$.
	\end{Le}
	\begin{proof}
		
		Let $\beta$ be a multi-index (we assume that $\beta$ may be zero). We use the notation
		\begin{equation}
			\partial^\beta f=\frac{\partial^\beta f}{\partial x^\beta}.
		\end{equation}
		One may see that
		\begin{equation}\label{prod}
			\int\limits_{\mathbb{R}^d}\partial^\beta\varphi_{x_1,r_1}(y)\partial^\beta\varphi_{x_2,r_2}(y)dy\lesssim \min(r_1,r_2)^dr_1^{-|\beta|}r_2^{-|\beta|}
		\end{equation}
		and
		\begin{equation}
			\int\limits_{\mathbb{R}^d}\partial^\beta\varphi_{x_1,r_1}(y)\partial^\beta\varphi_{x_2,r_2}(y)dy=0 \ \ \ \ \ \text{if } |x_1-x_2|\geqslant 3(r_1+r_2).
		\end{equation}
		First we prove~\eqref{L_2}.
		We can write the inequality
		\begin{equation}
			\int\limits_{\mathbb{R}^d}\varphi_{x_j,r_j}(y)\underset{r_k\leqslant r_j}{\sum\limits_{B_{r_k}(x_k)\in\mathfrak{B}}}\varphi_{x_k,r_k}(y)dy\lesssim
			\underset{B_{r_k}(x_k)\subset B_{7r_j}(x_j)}{\sum\limits_{B_{r_k}(x_k)\in\mathfrak{B}}}r_k^d\lesssim
			r_j^{d}.
		\end{equation}
		We finish the proof of~\eqref{L_2} by the estimate 
		
		\begin{equation}
			\begin{aligned}
				\Big\|\sum\limits_{B_{r_j}(x_j)\in \mathfrak{B}}\varphi_{x_j,r_j} \Big\|^2_{L_2(\mathbb{R}^d)}
				&=
				\int\limits_{\mathbb{R}^d}\Big(\sum\limits_{B_{r_j}(x_j)\in \mathfrak{B}}\varphi_{x_j,r_j}(y)\Big)^2dy\\
				&\leqslant
				2\int\limits_{\mathbb{R}^d}\sum\limits_{B_{r_j}(x_j)\in \mathfrak{B}}\Big(\varphi_{x_j,r_j}(y)\underset{r_k\leqslant r_j}{\sum\limits_{B_{r_k}(x_k)\in\mathfrak{B}}}\varphi_{x_k,r_k}(y)\Big)dy
				\lesssim
				\sum\limits_{B_{r_j}(x_j)\in \mathfrak{B}}r_j^d.
			\end{aligned}
		\end{equation}
		Now we prove~\eqref{W^d}.				
		Firstly we estimate the number of balls $B_{r_k}(x_k)\in\mathfrak{B}$ such that $2^nr_j\leqslant r_k<2^{n+1}r_j$ and $|x_j-x_k|\leqslant 3(r_k+r_j)$. Each such ball is a subset of the ball $B_{14\cdot2^nr_j}(x_j)$, so there exist at most $C=14^d$ such balls. 
		Using~\eqref{prod}, we can write the inequality
		\begin{equation}
			\int\limits_{\mathbb{R}^d}\partial^\beta\varphi_{x_j,r_j}(y)\underset{r_k\geqslant r_j}{\sum\limits_{B_{r_k}(x_k)\in\mathfrak{B}}}\partial^\beta\varphi_{x_k,r_k}(y)dy\lesssim
			r_j^{d-|\beta|}\sum\limits_{n=1}^\infty C(r2^n)^{-\beta}\lesssim r^{d-2|\beta|}.
		\end{equation}
		
		We finish the proof of~\eqref{W^d} by the estimate 
		\begin{equation}
			\begin{aligned}
				\int\limits_{\mathbb{R}^d}\Big(\sum\limits_{B_{r_j}(x_j)\in \mathfrak{B}}\partial^\beta\varphi_{x_j,r_j}(y)\Big)^2dy
				&\leqslant
				2\int\limits_{\mathbb{R}^d}\sum\limits_{B_{r_j}(x_j)\in \mathfrak{B}}\Big(\partial^\beta\varphi_{x_j,r_j}(y)\underset{r_k\geqslant r_j}{\sum\limits_{B_{r_k}(x_k)\in\mathfrak{B}}}\partial^\beta\varphi_{x_k,r_k}(y)\Big)dy\\
				&\lesssim
				\sum\limits_{B_{r_j}(x_j)\in \mathfrak{B}}r_j^{d-2|\beta|}.
			\end{aligned}
		\end{equation}
	\end{proof}
%

	\begin{Cor}\label{corbes}
		Let~$\varphi\in C^\infty_{0}(\mathbb{R}^d)$ be a bounded radially symmetric, radially non-increasing compactly supported function such that~$\supp \varphi\subset\overline{B}_3(0)$ and $0<s<\frac{d}{2}$. Then
		\begin{equation}
			\Big\|\sum\limits_{B_{r_j}(x_j)\in \mathfrak{B}}\varphi_{x_j,r_j} \Big\|_{B^{s}_{2,h}(\mathbb{R}^d)}\lesssim \|\{r_j| \ {B_{r_j}(x_j)\in \mathfrak{B}}\}\|^\frac{d-2s}{2}_{\ell_{d-2s,\frac{h(d-2s)}{2}}}
		\end{equation}
		for any disjoint family of balls~$\mathfrak{B}$  with radii of balls less than $1$.
	\end{Cor}
	\begin{proof}
		The corollary follows from Lemmas~\ref{strash},~\ref{dd} and the interpolation formula~\eqref{intS}.
	\end{proof}

	\begin{Th}\label{besres}
		Let $1\leqslant q<\infty$ and $0<s<\frac{d}{2}$. Let $\mu$ be  a measure of bounded variation such that $\mu\in B^{-s}_{2,q}(\mathbb{R}^d)$. Then for every set $A\subset\mathbb{R}^d$ such that $\mathcal{H}_{d-2s,\frac{q'}{2}}(A)<\infty$ we have $\mu(A)=0$.
	\end{Th}
	\begin{proof}
		The theorem follows from Corollaries~\ref{dualeq},~\ref{corbes} and dual formula~\eqref{B1.1}.
	\end{proof}

	\begin{Cor}\label{besdi}
		Let $1< q<\infty$ and $0<s<\frac{d}{2}$. Let $S\subset\mathbb{R}^d$ be a compact set such that $S$ is $\sigma$-finite relatively to $\mathcal{H}_{d-2s,\frac{q'}{2}}$. Let $\zeta$ be a distribution such that $\zeta\in B^{-s}_{2,q}(\mathbb{R}^d)$ and $\supp\zeta\subset{S}$. Then $\zeta=0$.
	\end{Cor}
	\begin{proof}
		The corollary follows from Theorem~\ref{besres} and Corollary~\ref{eqvi}.
	\end{proof}
	\begin{Cor}\label{capbes2}
		Let $1< q<\infty$ and $0<s<\frac{d}{2}$. Let $S\subset\mathbb{R}^d$ be a compact set such that $S$ is $\sigma$-finite relatively to $\mathcal{H}_{d-2s,\frac{q}{2}}$. Then $cap(S,B^{s}_{2,q}(\mathbb{R}^d))=0$.
	\end{Cor}	
	\begin{Rem}\label{vN}
		The inequality $cap(S,B^{s}_{2,q}(\mathbb{R}^d))\lesssim \mathcal{H}_{d-2s,\frac{q}{2}}(S)$ was proved by Netrusov (see~\cite[Remark to Statement~2.1]{Net}). From this inequality one can obtain Theorem~\ref{Netcap} by using Proposition~\ref{dual}.
	\end{Rem}
	
	\begin{proof}[Proof of Theorems~\ref{nonendmy} and~\ref{merm}]
		Let $\hat{\xi}\in L_{p,q}(\mathbb{R}^d)$ and $\supp \xi\in S$, where $S$ is $\sigma$-finite relatively to $\mathcal{H}_{\frac{2d}{p},\frac{q'}{2}}$. Then, by item~\ref{B4} of Proposition~\ref{BesPr} we have $\xi\in B_{2,q}^{\frac{d}{p}-\frac{d}{2}}$. Therefore by Corollary~\ref{besdi}, $\xi=0$. The proof of Theorem~\ref{merm} is similar, we only need to replace Corollary~\ref{besdi} by Theorem~\ref{besres}.
	\end{proof}

	\section{Examples of measures $\hat{\mu}\in L_{p,q}$}\label{S3}

	The construction follows the lines of the corresponding construction for Lebesgue spaces presented in our previous paper~\cite{Mynotyet}. To generalize to Lorentz spaces, we need to add some technical details.
	In Subsection~\ref{Constr} we provide the construction of $S$ and $\mu$ introduced in Theorem~\ref{resL}.
	In Subsection~\ref{ESC}  we prove that $\mu$ satisfies the relation $\hat{\mu}\in L_{p,q}$.
	
	In this chapter we will use the notation  $\lambda_Q$ for the Lebesgue probability measure on the cube $Q$ and also write $\lambda_0=\lambda_{[0,1]^d}$ for brevity.

	\subsection{A general Cantor--type set}\label{Constr}
	
	 The description naturally splits into two parts. First, we describe general Cantor type sets. Each such set corresponds to a sequence of cubes equipped with a tree structure. The sizes of the cubes are adjusted to fulfill the dimension condition. This construction still has freedom to shift smaller cubes inside larger ones. Second, we randomly choose these shifts to make the Fourier transform of the corresponding measure as small as prescribed by the mathematical expectation. 
	 In this subsection we fix $p$ and $\beta$.
	
	\subsubsection{General construction}
	Let $\mathcal{M}=\{M_0,M_1,...\}$ be an infinite sequence of natural numbers to be specified later. It will be rapidly increasing. For now we only assume that $M_j\geqslant 2$.
	We start with the construction of an infinite tree $\mathcal{T}$. We will inductively construct its subtrees $\mathcal{T}_k$ such that $\mathcal{T}_k\subset\mathcal{T}_{k+1}$ and set $\mathcal{T}=\cup\mathcal{T}_k$.
	
	Let $V(G)$ be the set of vertices of the graph $G$ and let $E(G)$ be the set of edges. Set $V(\mathcal{T}_0)=\{Q_0,Q_1,\dots,Q_{M_0}\}$ and $E(\mathcal{T}_0)=\{(Q_0,Q_1),(Q_0,Q_2),\dots,(Q_0,Q_{M_0})\}$. Here $Q_j$ are some cubes in $\mathbb{R}^d$ to be specified later.
	
	Assume we have constructed $\mathcal{T}_{k-1}$. To build $\mathcal{T}_k$, we add  $M_k$ new vertices to $\mathcal{T}_{k-1}$ and connect them with $Q_k$:
	\begin{align*}
		V(\mathcal{T}_k)&=V(\mathcal{T}_{k-1})\cup\{Q_{M_0+\dots+M_{k-1}+1},Q_{M_0+\dots+M_{k-1}+2},\dots,Q_{M_0+\dots+M_{k-1}+M_k}\},\\
		E(\mathcal{T}_k)&=E(\mathcal{T}_{k-1})\cup\{(Q_k,Q_{M_0+\dots+M_{k-1}+1}),(Q_k,Q_{M_0+\dots+M_{k-1}+2}),\dots,(Q_k,Q_{M_0+\dots+M_{k-1}+M_k})\}.
	\end{align*}
	We will say that $Q_k$ is the parent of those new vertices and that they are kids of $Q_k$.
	
	\begin{figure}[h]
		\center{\includegraphics[scale=0.915]{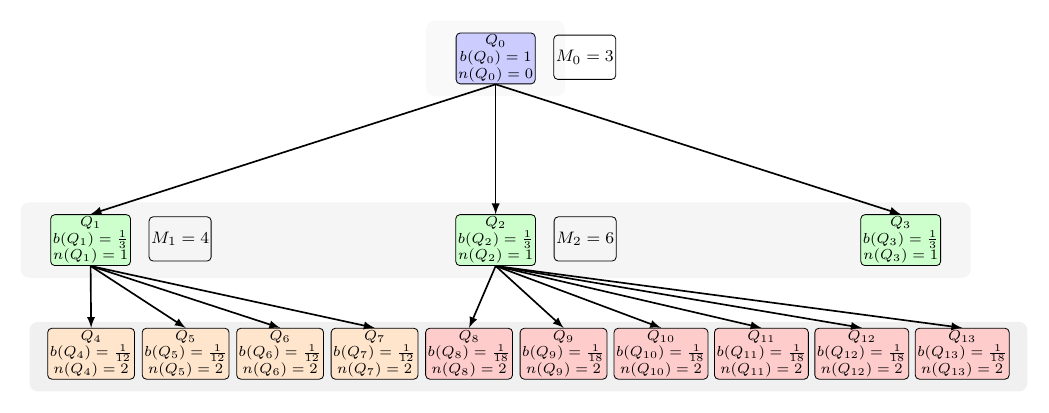}}
		\caption{Tree $\mathcal{T}_2$ for $M_0=3$, $M_1=4$, $M_2=6$}\label{pic2}
	\end{figure}
	We define the weight of the vertex $Q_i$ by the formula 
	\begin{equation}
		b(Q_i)=\!\!\!\!\underset{\text{\tiny{ancestor of }}Q_i}{\prod\limits_{Q_j\text{\tiny{ is an}}}}\!\!\!\! M_j^{-1}.
	\end{equation}

	We say that $Q_k$ belongs to the $n^{\text{th}}$ layer if it has exactly $n$ ancestors. The notation $n(Q_k)$ means the number of the layer of $Q_k$. 
	Clearly,
	\begin{equation}\label{probsum}
		\sum\limits_{\scriptscriptstyle{n(Q_k)=n}}b(Q_k)=1.
	\end{equation}
	Also we have the bound 
	\begin{equation}\label{2^n}
		b(Q_k)\leqslant 2^{-n(Q_k)}.
	\end{equation}
	We remind that $Q_k$ (the vertices) are cubes in $\mathbb{R}^d$. Let $l(Q)$ be the side length of the cube $Q$.
	\begin{Def}
		The cube sequence $\{Q_0,Q_1, \dots\}$  corresponds to the tree $\mathcal{T}$ if it satisfies the requirements:
		\begin{enumerate}[1)]
			\item $Q_0=[0,1]^d$;
			
			\item if $Q_i$ is a parent of $Q_j$, then $Q_j\subset Q_i$;
			
			\item if $Q_j$ is a kid of $Q_k$, then
			\begin{equation}\label{soot}
				l(Q_j)=r_k=\frac{M_k^{-\frac{p}{2d}}b(Q_k)^{\frac{p}{2d\beta}}}{n(Q_k)+1}.
			\end{equation}
		\end{enumerate} 
	\end{Def}
	Note that different kids of a cube may intersect.
	
	Formula~\eqref{soot} relates the size of the cubes and their number. If we consider the similar formula with denominator removed and $\beta=1$, then we will obtain a relation that allows us to say that the limit set has finite $\frac{2d}{p}$-Hausdorff measure. The parameter $\beta$ enables us to move from Hausdorff measure to the Netrusov--Hausdorff capacity. The denominator makes the limit set have zero Netrusov--Hausdorff capacity. See Proposition~\ref{H=0} below for the proof of these claims.
	\begin{Rem}
		Let the sequence $\mathcal{M}$ grow sufficiently fast ($M_k$ is much larger than all previous members of the sequence). Then the inequality $r_{k+1}<\frac{r_k}{2}$ is true for all integer $k$. Indeed if $M_0,\dots,M_{k-1}$ are fixed and $M_k$ tends to infinity, then $r_k\rightarrow 0$.
	\end{Rem}
	
	\begin{Def}
		Assume $\{Q_0,Q_1,\dots\}$  corresponds to $\mathcal{T}$. The set $C_n$ is defined by the formula
		\begin{equation}
			C_n=\underset{\scriptscriptstyle{n(Q_i)=n}}{\bigcup}Q_i.
		\end{equation}
		We call the set $C=\cap_{n=1}^\infty C_n$ the Cantor-type set  corresponding to $\mathcal{T}$.
	\end{Def}
	\begin{St}\label{H=0}
		Let $C$ be a Cantor-type set corresponding to $\mathcal{T}$. Assume $r_{k+1}<\frac{r_k}{2}$ for all $k$. Then $\mathcal{H}_{\frac{2d}{p},\beta}(C)=0$.
	\end{St}
	\begin{proof}
		The cubes of the $n^\text{th}$ layer provide a covering of $C$. We use those coverings to estimate the Netrusov--Hausdorff capacity of $C$:
		\begin{equation}
			\begin{aligned}
				\sum\limits_{l\in \mathbb{Z}}\Big(\underset{\diam Q_k\in\Delta_l}{\sum\limits_{\scriptscriptstyle{n(Q_k)=n}}}(\diam Q_k)^{\frac{2d}{p}}\Big)^\beta
				&\lesssim
				\sum\limits_{l\in \mathbb{Z}}\Big(\underset{\diam Q_k\in\Delta_l}{\sum\limits_{\scriptscriptstyle{n(Q_k)=n}}}l(Q_k)^{\frac{2d}{p}}\Big)^\beta\\
				&\!\!\!\!\!\!\!\overset{r_{k+1}< \frac{r_k}{2}}{=}
				\sum\limits_{\scriptscriptstyle{n(Q_k)=n-1}}M_k^{\beta}r_k^{\frac{2d\beta}{p}}
				\overset{\scriptscriptstyle{\text{\eqref{soot}}}}{=}
				\sum\limits_{\scriptscriptstyle{n(Q_k)=n-1}}\frac{b(Q_k)}{n^{\frac{2d\beta}{p}}}\ 
				\overset{\scriptscriptstyle{\eqref{probsum}}}{=}\ 
				\frac{1}{n^{\frac{2d\beta}{p}}}.
			\end{aligned}
		\end{equation}
		This quantity tends to zero as $n\rightarrow \infty$.
	\end{proof}
	To each Cantor-type set $C$, we will assign a probability measure $\mu$ such that $\supp\mu\subset C$. 	
	Denote the measure $\mu_k$ by the formula:
	\begin{equation}
		\mu_k=\sum\limits_{i=k}^{\scriptscriptstyle{M_0+\dots+M_{k-1}}}b(Q_i)\lambda_{Q_i}.
	\end{equation}
	For example, if $M_0=3$, $M_1=4$, $M_2=6$ (like in Figure~\ref{pic2}) 
	\begin{equation*}
		\mu_0=\lambda_{0},
	\end{equation*} 
	\begin{equation*}
		\mu_1=\frac{1}{3}(\lambda_{Q_1}+\lambda_{Q_2}+\lambda_{Q_3}),
	\end{equation*}
	\begin{equation*}
		\mu_2=\frac{1}{3}(\lambda_{Q_2}+\lambda_{Q_3})+\frac{1}{12}(\lambda_{Q_4}+\lambda_{Q_5}+\lambda_{Q_6}+\lambda_{Q_7}),
	\end{equation*}
	\begin{equation*}
		\mu_3=\frac{1}{3}\lambda_{Q_3}+\frac{1}{12}(\lambda_{Q_4}+\lambda_{Q_5}+\lambda_{Q_6}+\lambda_{Q_7})+\frac{1}{18}(\lambda_{Q_8}+\lambda_{Q_9}+\lambda_{Q_{10}}+\lambda_{Q_{11}}+\lambda_{Q_{12}}+\lambda_{Q_{13}}).
	\end{equation*}
	
	These measures  satisfy recurrence relations:
	\begin{equation}
		\mu_{0}=\lambda_{[0,1]^d},
	\end{equation}
	\begin{equation}\label{nonrandomrel}
		\mu_{k+1}=\mu_{k}-b(Q_k)\lambda_{Q_k}+\frac{b(Q_k)}{M_k}\underset{\text{\tiny{kid of }}Q_k}{\sum\limits_{Q_j \text{\tiny{ is a}}}}\lambda_{Q_j}.
	\end{equation}
	The measure $\mu$ is the weak* limit of the measures $\mu_k$. 
	\subsubsection{Random construction}
	The behavior of the Fourier transform of a Cantor measure can be quite chaotic (see~\cite{Strichartz} for examples). In this subsection, we will use randomization to choose a representative for which we can estimate $\hat{\mu}$. Note we have not specified the disposition of the ancestor cubes inside the parental cube. Now it is time to do that.

	Let $S_v\mu$ is a shift of a measure $\mu$ by a vector $v$. We mean that $S_v\mu(A)=\mu(A-v)$ for any Borel set $A$.
	\begin{Def}
		Let $M\in\mathbb{N}$ and let $0<r<\frac{1}{2}$. Let $\mu_{M,r}$ be the random variable taking values in the set of probability measures:
		\begin{equation}
			\mu_{M,r}=\frac{1}{M}\sum\limits_{j=1}^{M}S_{v_j}\lambda_{[0,r]^d}.
		\end{equation}
		Here $\{v_j\}_{j=1}^{M}$ is a sequence of independent vectors that are uniformly distributed on the cube $[0,1-r]^d$. 
	\end{Def}
	Note that different shifts of a cube $[0,r]^d$ may intersect.

	Let $f_r$ be the piecewise linear function (see Figure~\ref{pic1}) 
	
	\begin{equation}
		f_r(t)=
		\begin{cases}
			0, & t\in (-\infty,0]\cup[1,\infty);\\
			\frac{t}{r(1-r)}, & t\in [0,r];\\
			\frac{1}{1-r}, & t\in [r,1-r];\\
			\frac{1-t}{r(1-r)}, & t\in [1-r,1]. 
		\end{cases}
	\end{equation}
	\begin{figure}[h]
		\center{\includegraphics[scale=0.5]{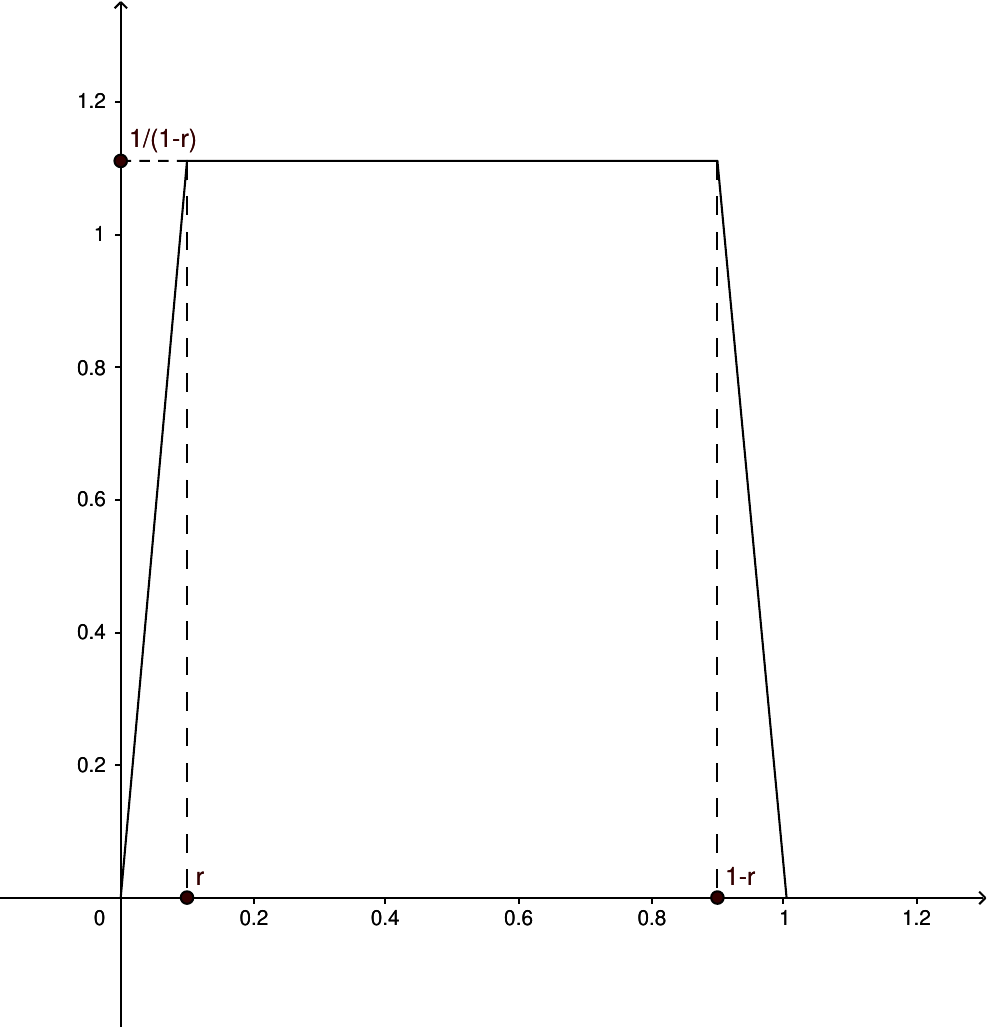}}
		\caption{Graph of $f_r$ for $r=0.1$}\label{pic1}
	\end{figure}
	
	Let $F_r(x)=\prod\limits_{i=1}^df_r(x_i)$ (here $x=(x_1,\dots,x_d)$). Then,
	\begin{equation}\label{mathex}
		\mathbb{E}\mu_{M,r}=\lambda_{[0,r]^d}*\lambda_{[0,1-r]^d}=F_r(x)dx
	\end{equation}
	and  
	
	\begin{equation}\label{glav}
		\hat{\mu}_{M,r}(x)\overset{D}{=}\frac{1}{M}\sum_{j=1}^{M}e^{-2\pi i<\beta_{r,j},x>}\hat{\lambda}_0(xr),
	\end{equation}

	Where $\{\beta_{r,j}\}_{j=1}^{M}$ is the sequence of independent vectors that are uniformly distributed on the cube $[0,1-r]^d$. The notation $\overset{D}{=}$ means equality of distributions.
	
	Let $\nu_{M,r}=\mu_{M,r}(\omega)$ be the value of $\mu_{M,r}$  at some point $\omega$ of the probability space to be chosen later (see Corollary~\ref{PrE}). The choice of $\omega$ defines the disposition of smaller cubes inside the larger one. By definition, $\nu_{M,r}$ is a probability measure.
	
	We will use the measures $\nu_{M,r}$ to construct the sequences $\{Q_0,\dots,Q_{M_0+\dots+M_k}\}$ of cubes inductively. Set $Q_0=[0,1]^d$. The measure $\nu_{M_0,r_0}$ corresponds to $M_0$  cubes. The sequence $\{Q_1,\dots,Q_{M_0}\}$ consists of those cubes.
	
	Assume we have constructed the cubes $Q_0,Q_1,\dots,Q_{M_0+M_1+\dots+M_{k-1}}$. Let $T_{Q_k}$ be the homothety  that maps $[0,1]^d$ onto $Q_k$. If $Q_k=y+[0,l(Q_k)]^d$, then $T_{Q_k}(x)=y+l(Q_k)x$.
	Let 
	\begin{equation}
		\nu_{M_k,\frac{r_k}{l(Q_k)},Q_k}
		=
		T_{Q_k}(\nu_{M_k,\frac{r_k}{l(Q_k)}}),
	\end{equation}
	we mean that
	\begin{equation}
		\nu_{M_k,\frac{r_k}{l(Q_k)},Q_k}(A)
		=
		\nu_{M_k,\frac{r_k}{l(Q_k)}}(T^{-1}_{Q_k}(A)),
	\end{equation}
	for any Borel set $A$.
	The measure $\nu_{M_k,\frac{r_k}{l(Q_k)},Q_k}$ corresponds to $M_k$  cubes (see Figure~\ref{pic}). Add them to the end of the cube sequence. Thus, we have constructed the cubes $Q_{M_0+M_1+\dots+M_{k-1}+1},\dots,Q_{M_0+M_1+\dots+M_{k}}$.

	\begin{figure}[h]
		\center{\includegraphics[scale=0.5]{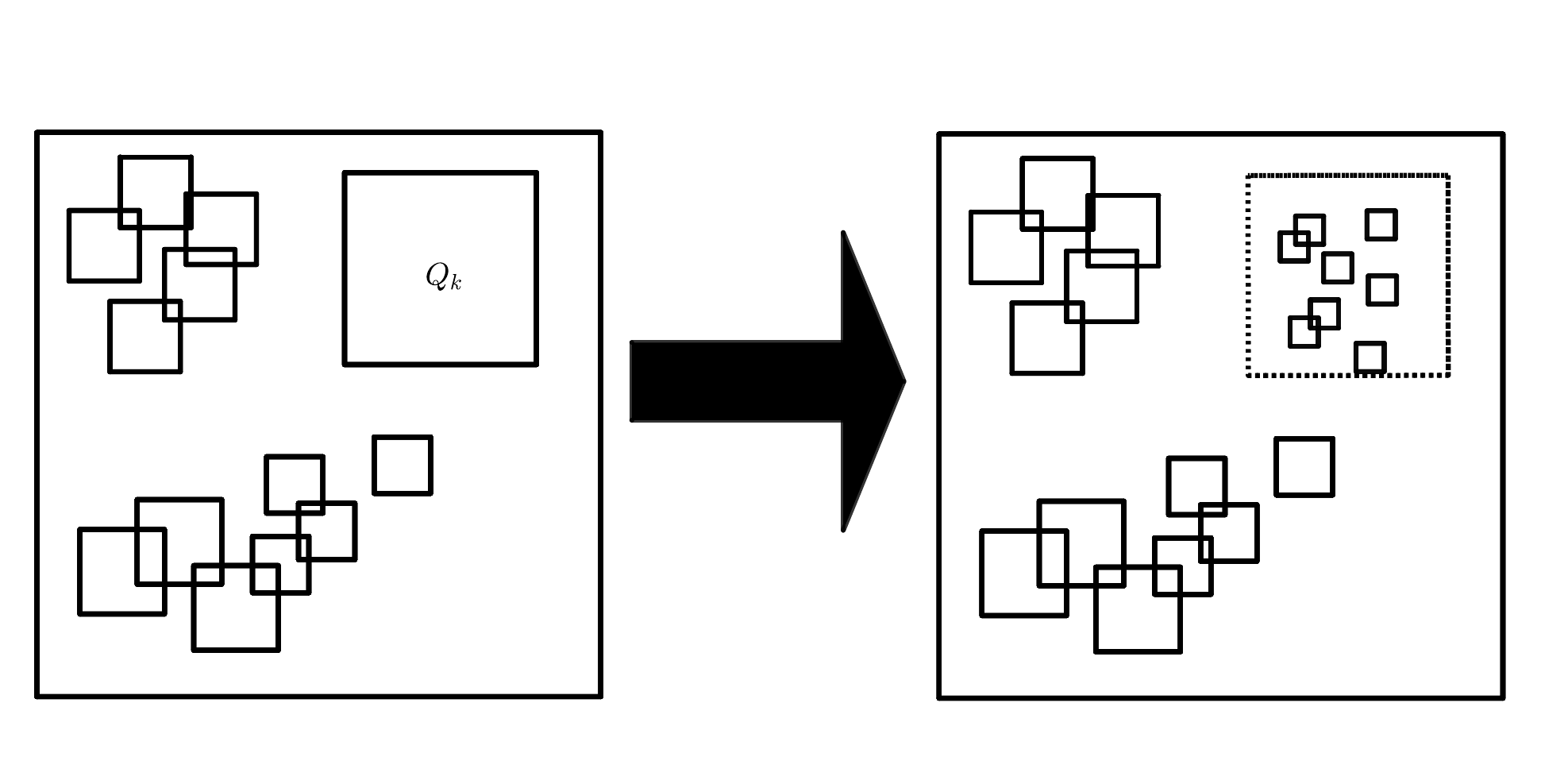}}
		\caption{Construction}\label{pic}
	\end{figure}
	
	\begin{Rem}
		The cube sequence $\{Q_0,Q_1,\dots\}$ corresponds to a tree $\mathcal{T}$.
	\end{Rem}
	
	For this cube sequence the recurrence relations~\eqref{nonrandomrel} turn into
	
	\begin{equation}\label{formz}
		\mu_{k+1}=\mu_{k}-b(Q_k)\lambda_{Q_k}+b(Q_k)\nu_{M_k,\frac{r_k}{l(Q_k)},Q_k}.
	\end{equation}

	\subsection{Estimation of example}\label{ESC}

	The following lemma was proved in~\cite[Lemma~3.2]{Mynotyet}.
	\begin{Le}\label{Np}
		For all $M\in \mathbb{N}$, $r<\frac{1}{2}$ and $p>1$ the inequality
		\begin{equation}
			\int\limits_{\mathbb{R}^d}\mathbb{E}|\hat{\mu}_{M,r}(x)-\mathbb{E}\hat{\mu}_{M,r}(x)|^pdx\lesssim M^{-\frac{p}{2}}r^{-d}
		\end{equation}
		is true.
	\end{Le}

	Fix $p_1$ and $p_2$ such that $p_1>p>p_2$.
	
	\begin{Cor}\label{PrE}
		For all $M\in\mathbb{N}$ and $0<r<\frac{1}{2}$ we can choose measures $\nu_{M,r}=\mu_{M,r}(\omega)$ such that
		
		\begin{equation}
			\int\limits_{\mathbb{R}^d}|\hat{\nu}_{M,r}(x)-\mathbb{E}\hat{\mu}_{M,r}(x)|^{p_1}dx\lesssim M^{-\frac{p_1}{2}}r^{-d},
		\end{equation}
		\begin{equation}
			\int\limits_{\mathbb{R}^d}|\hat{\nu}_{M,r}(x)-\mathbb{E}\hat{\mu}_{M,r}(x)|^{p_2}dx\lesssim M^{-\frac{p_2}{2}}r^{-d}.
		\end{equation}
	\end{Cor}
	
	\begin{Le}\label{ooo}
		Let $p>2$. Then the inequality
		\begin{equation}
			\|\hat{\lambda}_0-\mathbb{E}\hat{\mu}_{M,r}\|_{L_{p}}\lesssim r^{\frac{1}{p'}}
		\end{equation}
		is true.
	\end{Le}
	\begin{proof}
		One may see from~\eqref{mathex} that $\|\lambda_0-\mathbb{E}\mu_{M,r}\|_{L_{p'}}\lesssim r^{\frac{1}{p'}}$. The Hausdorff--Young inequality finishes the proof.
	\end{proof}
		
	\begin{Cor}\label{leb}
		Let $p>2$. Then the inequalities
		\begin{equation}\label{il}
			\|\hat{\lambda}_0-\mathbb{E}\hat{\mu}_{M,r}\|_{L_{p,q}}\lesssim r^{\frac{1}{p'}},
		\end{equation}
		\begin{equation}\label{pq}
			\|\hat{\nu}_{M,r}-\mathbb{E}\hat{\mu}_{M,r}\|_{L_{p,q}}\lesssim M^{-\frac{1}{2}}r^{-\frac{d}{p}}
		\end{equation}
		are true.
		Here $\nu_{M,r}$ is defined by the choice of $\omega$ in Corollary~\ref{PrE}.
	\end{Cor}
	\begin{proof}
		By Corollary~\ref{PrE}, we have
		\begin{equation}
			\|\hat{\nu}_{M,r}-\mathbb{E}\hat{\mu}_{M,r}\|_{L_{p_1}}\lesssim M^{-\frac{1}{2}}r^{-\frac{d}{p_1}},
		\end{equation}
		\begin{equation}
			\|\hat{\nu}_{M,r}-\mathbb{E}\hat{\mu}_{M,r}\|_{L_{p_2}}\lesssim M^{-\frac{1}{2}}r^{-\frac{d}{p_2}}.
		\end{equation}
		Items~\ref{I1} and~\ref{I4} of Proposition~\ref{inter} finish proof of~\eqref{pq}. The inequality~\eqref{il} is derived from Lemma~\ref{ooo} in a similar way.
	\end{proof}
	
	\begin{Cor}\label{cor1}
		We have the inequalities
		\begin{equation}
			\|\hat{\nu}_{M,r}-\hat{\lambda}_0\|^q_{L_{p,q}}\lesssim M^{-\frac{q}{2}}r^{-\frac{qd}{p}} + r^{\frac{q}{p'}},
		\end{equation}
		\begin{equation}
			\|\hat{\nu}_{M,r}-\hat{\lambda}_0\|^{p_1}_{L_{p_1}}\lesssim M^{-\frac{p_1}{2}}r^{-d}+r^{\frac{p_1}{p_1'}}.
		\end{equation}
		Here $\nu_{M,r}$ is defined by the choice of $\omega$ in Corollary~\ref{PrE} and $p_1$ was fixed before.
	\end{Cor}

	\begin{Le}\label{Ras}
		Let $M_0,M_1,M_2,\dots,M_{k-1}$ be fixed and let $M_k$ tend to infinity. Then, the inequalities
		\begin{equation}
			\varlimsup\limits_{M_k\rightarrow\infty}\|\hat{\mu}_{k+1}-\hat{\mu}_{k}\|^q_{L_{p,q}}\lesssim b(Q_k)^{q-\frac{q}{2\beta}}(n(Q_k)+1)^{\frac{qd}{p}},
		\end{equation}
		\begin{equation}
			\lim\limits_{M_k\rightarrow\infty}\|\hat{\mu}_{k+1}-\hat{\mu}_{k}\|_{L_{p_1}}=0
		\end{equation}
		are true.
	\end{Le}
	\begin{proof}
		We can write the estimates
		\begin{equation}
			\begin{aligned}
				\|\hat{\mu}_{k+1}-\hat{\mu}_{k}\|^q_{L_{p,q}}\ &\!\!\!\overset{\scriptscriptstyle{\text{\eqref{formz}}}}{=}
				\|b(Q_k)(\hat{\nu}_{M_k,\frac{r_k}{l(Q_k)},Q_k}-\hat{\lambda}_{Q_k})\|^q_{L_{p,q}}\\
				&=				b(Q_k)^q\|\hat{\nu}_{M_k,\frac{r_k}{l(Q_k)},Q_k}-\hat{\lambda}_{Q_k}\|^q_{L_{p,q}}\\
				&=
				b(Q_k)^ql(Q_k)^{-\frac{qd}{p}}\|\hat{\nu}_{M_k,\frac{r_k}{l(Q_k)}}-\hat{\lambda}_0\|^q_{L_{p,q}}\\
				&\!\!\!\!\overset{\scriptscriptstyle{\text{Cor~\ref{cor1}}}}{\lesssim}
				b(Q_k)^ql(Q_k)^{-\frac{qd}{p}}\left(M_k^{-\frac{q}{2}}\left(\frac{r_k}{l(Q_k)}\right)^{-\frac{qd}{p}} + \left(\frac{r_k}{l(Q_k)}\right)^{\frac{q}{p'}}\right)\\
				&\overset{\scriptscriptstyle{\text{\eqref{soot}}}}{=}
				b(Q_k)^{q}l(Q_k)^{-\frac{qd}{p}}\left((n(Q_k)+1)^{\frac{qd}{p}}b(Q_k)^{-\frac{q}{2\beta}}l(Q_k)^{\frac{qd}{p}} + \left(\frac{r_k}{l(Q_k)}\right)^{\frac{q}{p'}}\right)\\
				&\overset{\scriptscriptstyle{M_k\rightarrow\infty}}{\longrightarrow}
				b(Q_k)^{q-\frac{q}{2\beta}}(n(Q_k)+1)^{\frac{qd}{p}},
			\end{aligned}
		\end{equation}
		and in the same way we have
		\begin{equation}
			\|\hat{\mu}_{k+1}-\hat{\mu}_{k}\|^{p_1}_{L_{p_1}}
			\lesssim
			b(Q_k)^{p_1-\frac{p}{2}}M_k^{-\frac{p_1-p}{2}}(n(Q_k)+1)^d+
			b(Q_k)^{p_1}l(Q_k)^{-d}\left(\frac{r_k}{l(Q_k)}\right)^{\frac{p_1}{p_1'}}
			\overset{\scriptscriptstyle{M_k\rightarrow\infty}}{\longrightarrow}
			0.
		\end{equation}
		Here  $b(Q_k)$, $n(Q_k)$ and $l(Q_k)$ are fixed, and $r_k\rightarrow0$ as $M_k\rightarrow\infty$ by~\eqref{soot}.
	\end{proof}

	\begin{Le}\label{p+p}
		Let $f\in L_{p,q}$ be a function and let $g_j\in L_{p,q}$ be a sequence of functions. Assume that $\varlimsup\limits_{j\rightarrow\infty}\|g_j\|_{L_{p,q}}\leqslant A$ and for some $p_1>p$ we have $\lim\limits_{j\rightarrow\infty}\|g_j\|_{L_{p_1}}=0$. Then
		\begin{equation}
			\varlimsup\limits_{j\rightarrow\infty} \|f+g_j\|^q_{L_{p,q}}\leqslant \|f\|^q_{L_{p,q}}+A^q.
		\end{equation} 
	\end{Le}
	We will prove this lemma in Section~\ref{tech}.

	\begin{Cor}
		There exists a constant $C>0$ such that the inequality
		\begin{equation}
			\|\hat{\mu}_{k+1}\|^q_{L_{p,q}}\leqslant \|\hat{\mu}_{k}\|^q_{L_{p,q}}+Cb(Q_k)^{q-\frac{q}{2\beta}}(n(Q_k)+1)^{\frac{qd}{p}}
		\end{equation}
		is true, provided $M_k$ is  sufficiently large.
	\end{Cor}
	This corollary is a combination of Lemmas~\ref{Ras} and~\ref{p+p}.

	\begin{proof}[Proof of Theorem~\ref{resL}]
		
		Let $\{M_0,M_2,\dots\}$ be a rapidly growing infinite sequence. Let $\mu_{\infty}$ be a weak* limit of the $\mu_{k}$. The rapid growth of  $M_k$ provides us with the following inequality:
		\begin{equation}
			\begin{aligned}
				\|\hat{\mu}_{k}\|^q_{L_{p,q}}&\leqslant \|\hat{\lambda}_0\|^q_{L_{p,q}}+C\sum\limits_{j=0}^{k-1}b(Q_j)^{q-\frac{q}{2\beta}}n(Q_j)^{\frac{qd}{p}}\\
				&\lesssim
				\sum\limits_{j=0}^{\infty}b(Q_j)^{q-\frac{q}{2\beta}}(n(Q_j)+1)^{\frac{qd}{p}}
				=
				\sum\limits_{n=0}^{\infty}\sum\limits_{n(Q_j)=n}b(Q_j)^{q-\frac{q}{2\beta}}(n+1)^{\frac{qd}{p}}\\
				&\overset{\eqref{2^n}}{\leqslant}
				\sum\limits_{n=0}^{\infty}\sum\limits_{n(Q_j)=n}b(Q_j)2^{-n(q-\frac{q}{2\beta}-1)}(n+1)^{\frac{qd}{p}}\ 
				\overset{\scriptscriptstyle{\eqref{probsum}}}{=} 
				\sum\limits_{n=0}^{\infty} 2^{-n\frac{q-1}{\beta}(\beta-\frac{q'}{2})}(n+1)^{\frac{qd}{p}}<\infty.
			\end{aligned}
		\end{equation}
		
		Thus $\hat{\mu}_{\infty}\in L_{p,q}(\mathbb{R}^d)$. This formula and Proposition~\ref{H=0} complete the proof of Theorem~\ref{resL}.
	\end{proof}

	\section{Proof of technical lemmas}\label{tech}
	
	\subsection*{Proof of Proposition~\ref{HLp}.}
	\
	
	{\bf Item 1} is trivial.
	
	{\bf Item 2} states that $\mathcal{H}_{\alpha,q}$ is an outer measure.
	Let $A\subset\cup A_j$. We want to prove that $\mathcal{H}_{\alpha,q}(A)\leqslant\sum \mathcal{H}_{\alpha,q}(A_j)$.
	
	{\bf Case $q\leqslant1$ or $q=\infty$.}
	Let $\mathfrak{D}_j$ be a covering of $A_j$ such that $\diam B_i<\delta$ for all $B_i\in\mathfrak{D}_j$ and
	\begin{equation}
		\sum\limits_k\Big(\underset{\diam B_i\in\Delta_k}{\sum\limits_{B_i\in\mathfrak{D}_j}}(\diam B_i)^\alpha\Big)^q\leqslant 
		\mathcal{H}_{\alpha,q}(A_j)+\frac{\varepsilon}{2^j} \ \ \ \ \ \ \text{in the case } q<\infty,
	\end{equation}
	\begin{equation}
		\sup\limits_k\underset{\diam B_i\in\Delta_k}{\sum\limits_{B_i\in\mathfrak{D}_j}}(\diam B_i)^\alpha\leqslant 
		\mathcal{H}_{\alpha,q}(A_j)+\frac{\varepsilon}{2^j} \ \ \ \ \ \ \text{in the case } q=\infty.
	\end{equation}
	Let $\mathfrak{D}=\cup\mathfrak{D}_j$. The family $\mathfrak{D}$ is a covering of $A$. We can write the inequality for $q \leqslant 1$ 
	\begin{equation}
		\begin{aligned}
			\sum\limits_k\Big(\sum\limits_{j}\underset{\diam B_i\in\Delta_k}{\sum\limits_{B_i\in\mathfrak{D}_j}}(\diam B_i)^\alpha\Big)^q
			&\leqslant 
			\sum\limits_k\sum\limits_{j}\Big(\underset{\diam B_i\in\Delta_k}{\sum\limits_{B_i\in\mathfrak{D}_j}}(\diam B_i)^\alpha\Big)^q\\
			&\leqslant
			\sum\limits_{j}\big(\mathcal{H}_{\alpha,q}(A_j)+\frac{\varepsilon}{2^j}\big)
			=
			\sum\limits_{j}\mathcal{H}_{\alpha,q}(A_j)+\varepsilon.
		\end{aligned}
	\end{equation}	
	Similarly, for $q=\infty$
	\begin{equation}
		\begin{aligned}
			\sup\limits_k\sum\limits_{j}\underset{\diam B_i\in\Delta_k}{\sum\limits_{B_i\in\mathfrak{D}_j}}(\diam B_i)^\alpha
			&\leqslant 
			\sum\limits_{j}\sup\limits_k\underset{\diam B_i\in\Delta_k}{\sum\limits_{B_i\in\mathfrak{D}_j}}(\diam B_i)^\alpha\\
			&\leqslant
			\sum\limits_{j}\big(\mathcal{H}_{\alpha,q}(A_j)+\frac{\varepsilon}{2^j}\big)
			=
			\sum\limits_{j}\mathcal{H}_{\alpha,q}(A_j)+\varepsilon.
		\end{aligned}
	\end{equation}
	
	{\bf Case $1<q<\infty$.}	
	Let $\mathfrak{D}_1$ be a covering of $A_1$ such that 
	\begin{equation}
		\sum\limits_k\Big(\underset{\diam B_i\in\Delta_k}{\sum\limits_{B_i\in\mathfrak{D}_1}}(\diam B_i)^\alpha\Big)^q\leqslant 
		\mathcal{H}_{\alpha,q}(A_1)+\frac{\varepsilon}{2}. 
	\end{equation}
	We will choose families $\mathfrak{D}_j$ inductively. Assume we have chosen  $\mathfrak{D}_1,\dots,\mathfrak{D}_N$ and numbers $k_1,\dots,k_N$. Let $k_{N+1}\in\mathbb{N}$ be such that $k_{N+1}>k_N$ and
	\begin{equation}\label{epsq}
		\sum\limits_{k=k_{N+1}}^{\infty}\Big(\sum\limits_{j=1}^{N}\underset{\diam B_i\in\Delta_k}{\sum\limits_{B_i\in\mathfrak{D}_j}}(\diam B_i)^\alpha\Big)^q\leqslant \frac{\varepsilon^q}{2^{N+1}}. 
	\end{equation}
	We choose $\mathfrak{D}_{N+1}$ to be a covering of $A_{N+1}$ such that for any $B\in\mathfrak{D}_{N+1}$ we have $\diam B<2^{-k_{N+1}-1}$ and 
	\begin{equation}
		\sum\limits_k\Big(\underset{\diam B_i\in\Delta_k}{\sum\limits_{B_i\in\mathfrak{D}_{N+1}}}(\diam B_i)^\alpha\Big)^q\leqslant 
		\mathcal{H}_{\alpha,q}(A_{N+1})+\frac{\varepsilon}{2}. 
	\end{equation}
	This choice provides us with the covering  $\mathfrak{D}=\cup\mathfrak{D}_j$ of $A$.
	
	Using the inequality $(a+b)^q\leqslant \frac{a^q}{\varepsilon^{q-1}}+\frac{b^q}{(1-\varepsilon)^{q-1}}$ for $q>1$, we may write the estimate
	\begin{equation}
		\begin{aligned}
			&\sum\limits_k\Big(\underset{\diam B_i\in\Delta_k}{\sum\limits_{B_i\in\mathfrak{D}}}(\diam B_i)^\alpha\Big)^q=
			\sum\limits_k\Big(\sum\limits_{j=1}^{\infty}\underset{\diam B_i\in\Delta_k}{\sum\limits_{B_i\in\mathfrak{D}_j}}(\diam B_i)^\alpha\Big)^q\\
			&=
			\sum\limits_{N}\sum\limits_{k_N+1}^{k_{N+1}}\Big(\sum\limits_{j=1}^{N}\underset{\diam B_i\in\Delta_k}{\sum\limits_{B_i\in\mathfrak{D}_j}}(\diam B_i)^\alpha\Big)^q\\
			&\leqslant
			\sum\limits_{N}\sum\limits_{k_N+1}^{k_{N+1}}\frac{1}{\varepsilon^{q-1}}\Big(\sum\limits_{j=1}^{N-1}\underset{\diam B_i\in\Delta_k}{\sum\limits_{B_i\in\mathfrak{D}_j}}(\diam B_i)^\alpha\Big)^q
			+
			\frac{1}{(1-\varepsilon)^{q-1}}\Big(\underset{\diam B_i\in\Delta_k}{\sum\limits_{B_i\in\mathfrak{D}_N}}(\diam B_i)^\alpha\Big)^q\\
			&\overset{\eqref{epsq}}{\leqslant}
			\sum_{N=1}^\infty\left[\frac{\mathcal{H}_{\alpha,q}(A_N)+\frac{\varepsilon}{2^N}}{(1-\varepsilon)^{q-1}}+\frac{\varepsilon}{2^N}\right].
		\end{aligned}
	\end{equation}
	Considering arbitrarily small $\varepsilon$, we get
	\begin{equation}
		\mathcal{H}_{\alpha,q}( A)
		\leqslant
		\sum\limits_{N=1}^\infty \mathcal{H}_{\alpha,q}(A_N).
	\end{equation}
	{\bf Item 3} states that if $1\leqslant q<\infty$, then $\mathcal{H}_{\alpha,q}$ is a Borel measure.
	
	From the previous item we know that $\mathcal{H}_{\alpha,q}$ is an outer measure. To prove that it is a Borel measure it suffices to show that $\mathcal{H}_{\alpha,q}$ is additive on separated sets (see~\cite[Theorem~1.7]{Mattila}). Let $A_1$ and $A_2$ be sets such that $\dist(A_1,A_2)>0$, we want to prove that $\mathcal{H}_{\alpha,q}(A_1\cup A_2)=\mathcal{H}_{\alpha,q}(A_1)+\mathcal{H}_{\alpha,q}(A_2)$. Let $\mathfrak{D}$ be a covering of $A_1\cup A_2$ such that the diameter of any set in $\mathfrak{D}$ is less than $\dist(A_1,A_2)>0$. Then the family $\mathfrak{D}$ can be split into two subfamilies $\mathfrak{D}_1$ and $\mathfrak{D}_2$ are coverings of $A_1$ and $A_2$ respectively. If the diameters of all sets are sufficiently small, then the inequality
	\begin{equation}
		\sum\limits_k\Big(\underset{\diam B_i\in\Delta_k}{\sum\limits_{B_i\in\mathfrak{D}_j}}(\diam B_i)^\alpha\Big)^q> 
		\mathcal{H}_{\alpha,q}(A_j)-\varepsilon 
	\end{equation}
	is true.
	Therefore,
	\begin{equation}
		\begin{aligned}
			&\sum\limits_k\Big(\underset{\diam B_i\in\Delta_k}{\sum\limits_{B_i\in\mathfrak{D}}}(\diam B_i)^\alpha\Big)^q\\
			&=
			\sum\limits_k\Big(\underset{\diam B_i\in\Delta_k}{\sum\limits_{B_i\in\mathfrak{D}_1}}(\diam B_i)^\alpha+\underset{\diam B_i\in\Delta_k}{\sum\limits_{B_i\in\mathfrak{D}_2}}(\diam B_i)^\alpha\Big)^q\\
			&\geqslant
			\sum\limits_k\Big(\underset{\diam B_i\in\Delta_k}{\sum\limits_{B_i\in\mathfrak{D}_1}}(\diam B_i)^\alpha\Big)^q+\sum\limits_k\Big(\underset{\diam B_i\in\Delta_k}{\sum\limits_{B_i\in\mathfrak{D}_2}}(\diam B_i)^\alpha\Big)^q\\
			&>
			\mathcal{H}_{\alpha,q}(A_1)+\mathcal{H}_{\alpha,q}(A_2)-2\varepsilon.
		\end{aligned}
	\end{equation}
	So, $\mathcal{H}_{\alpha,q}(A_1\cup A_2)\geqslant\mathcal{H}_{\alpha,q}(A_1)+\mathcal{H}_{\alpha,q}(A_2)$. This provides us with the equality $\mathcal{H}_{\alpha,q}(A_1\cup A_2)=\mathcal{H}_{\alpha,q}(A_1)+\mathcal{H}_{\alpha,q}(A_2)$, because $\mathcal{H}_{\alpha,q}$ is an outer measure.
	
	{\bf Item 4} states that if $S$ is a $\sigma$-finite set relative to the $\alpha$-packing measure, then, for any $q>0$, $S$ is $\sigma$-finite relative to  $\mathcal{H}_{\alpha,q}$.
	
	It suffices to prove that if $\widetilde{\mathcal{P}}_\alpha(A)<\infty$ then $\mathcal{H}_{\alpha,q}(A)<\infty$. 
	Let $\mathfrak{B}$ be a family of closed balls with centers in a maximal $\varepsilon$-separated subset of set $A$ and radii equal to $\varepsilon$. The family $\mathfrak{B}$ is a covering of $S$. Thus,
	
	\begin{equation}
		\sum\limits_k\Big(\underset{\diam B_i\in\Delta_k}{\sum\limits_{B_i\in\mathfrak{B}}}(\diam B_i)^\alpha\Big)^q
		=
		(N_\varepsilon(A)(2\varepsilon)^\alpha)^q
		\overset{Rem~\ref{yp}}{\lesssim}
		\widetilde{\mathcal{P}}_\alpha(A)^q<\infty. 
	\end{equation}

	{\bf Item 5} states that if $0<q_1<q_2$, then $\mathcal{H}_{\alpha,q_2}$ is absolutely continuous with respect to $\mathcal{H}_{\alpha,q_1}$.
	
	In this case, $\|. \|_{\ell_{\alpha,\infty}}\lesssim\|. \|_{\ell_{\alpha,\alpha q_2}}\lesssim\|. \|_{\ell_{\alpha,\alpha q_1}}$. Therefore, according to~\eqref{NHq} and~\eqref{NHi},
	\begin{equation}
		\mathcal{H}_{\alpha,\infty}^{\frac{1}{\alpha}}\lesssim
		\mathcal{H}_{\alpha,\alpha q_2}^{\frac{1}{\alpha q_2}}\lesssim
		\mathcal{H}_{\alpha,\alpha q_1}^{\frac{1}{\alpha q_1}},
	\end{equation}
	and, in particular, we have the desired absolutely continuity.
	
	{\bf Item 6} states that if $1\leqslant q_1<q_2$, then we have the following statements:
	
	If $\mathcal{H}_{\alpha,q_1}(S)<\infty$, then   $\mathcal{H}_{\alpha,q_2}(S)=0$;
	
	If $\mathcal{H}_{\alpha,q_2}(S)>0$, then   $\mathcal{H}_{\alpha,q_1}(S)=\infty$.
	
	It suffices to prove the first part of the item in the case $q_2<\infty$.
	Let $\mathcal{H}_{\alpha,q_1}(S)<\infty$. Split $S=\sqcup S_j$ such that  $\mathcal{H}_{\alpha,q_1}(S_j)<\frac{1}{n}$. We can write the inequality
	\begin{equation}
		\mathcal{H}_{\alpha,q_2}(S)
		=
		\sum\limits_{j}\mathcal{H}_{\alpha,q_2}(S_j)
		\lesssim
		\sum\limits_{j}\mathcal{H}_{\alpha,q_1}(S_j)^{\frac{q_2}{q_1}}
		\leqslant
		\frac{1}{n^{\frac{q_2}{q_1}-1}}\sum\limits_{j}\mathcal{H}_{\alpha,q_1}(S_j)
		=
		\frac{1}{n^{\frac{q_2}{q_1}-1}}\mathcal{H}_{\alpha,q_1}(S)
		\overset{n\rightarrow\infty}{\rightarrow}
		0.
	\end{equation}
	
	{\bf Item 7} states that if $\alpha_1<\alpha_2$, then for any $0<q_1,q_2$ we have the following statements:
	
	If $\mathcal{H}_{\alpha_1,q_1}(S)<\infty$, then   $\mathcal{H}_{\alpha_2,q_2}(S)=0$;
	
	If $\mathcal{H}_{\alpha_2,q_2}(S)>0$, then   $\mathcal{H}_{\alpha_1,q_1}(S)=\infty$.
	
	It suffices to prove the first part of the item in the case $q_1=\infty$.
	Let $\mathcal{H}_{\alpha_1,\infty}(S)<\infty$. It means that there exist coverings $\mathfrak{D}_j$ such that $\diam B_i\leqslant2^{-j}$ for all $B_i\in\mathfrak{D}_j$ and satisfying the estimate
	\begin{equation}
		\sup\limits_k\underset{\diam B_i\in\Delta_k}{\sum\limits_{B_i\in\mathfrak{D}_j}}(\diam B_i)^{\alpha_1}
		\leqslant 
		C.
	\end{equation}

	So we can estimate $\mathcal{H}_{\alpha_2,q_2}(S)$ by the quantity
	
	\begin{equation}
		\sum\limits_{k=j}^\infty\Big(\underset{\diam B_i\in\Delta_k}{\sum\limits_{B_i\in\mathfrak{D}_j}}(\diam B_i)^{\alpha_2}\Big)^{q_2}
		\leqslant
		\sum\limits_{k=j}^\infty(2^{-k(\alpha_2-\alpha_1)}C)^{q_2}
		=
		C^{q+2}\sum\limits_{k=j}^\infty2^{-k(\alpha_2-\alpha_1)q_2}
		\overset{j\rightarrow\infty}{\longrightarrow}
		0.
	\end{equation}
	
	{\bf Item 8} states that if $0<q<1$ and $f$ is a regular function such that
	\begin{equation}
		\int\limits_{0}^{1}\left(\frac{t^\alpha}{f(t)}\right)^{\frac{q}{1-q}}\frac{dt}{t}<\infty,
	\end{equation}
	then $\mathcal{H}_{\alpha,q}$ is absolutely continuous with respect to $\Lambda_f$.
	Moreover, if $\mathcal{H}_{\alpha,q}(S)>0$, then $\Lambda_f(S)=\infty$; if $\Lambda_f(S)<\infty$, then $\mathcal{H}_{\alpha,q}(S)=0$.

	It suffices to prove that if $\mathcal{H}_{\alpha,q}(S)>0$, then $\Lambda_f(S)=\infty$.
	Let $\Lambda_f(S)<\infty$. It means that there exist coverings $\mathfrak{D}_j$ such that the diameters of all sets are not greater than $2^{-j}$ and satisfy the estimate
	\begin{equation}
		\sum\limits_{B_i\in\mathfrak{D}_j}f(\diam B_i)\leqslant 
		C
	\end{equation}
	Let $M_k$ be the number of sets $B_i$ such that $\diam B_i\in\Delta_k$.		
	We can estimate $\mathcal{H}_{\alpha,q}(S)$ by the quantity
	\begin{equation}
		\begin{aligned}
			\sum\limits_{k=j}^\infty\Big(\underset{\diam B_i\in\Delta_k}{\sum\limits_{B_i\in\mathfrak{D}_j}}(\diam B_i)^{\alpha}\Big)^{q}
			&\asymp
			\sum\limits_{k=j}^\infty M_k^q2^{-qk\alpha}\\
			&\lesssim
			\Big(\sum\limits_{k=j}^\infty M_kf(2^{-k})\Big)^q
			\Big(\sum\limits_{k=j}^\infty \Big(\frac{2^{-k\alpha}}{f(2^{-k})}\Big)^{\frac{q}{1-q}}\Big)^{1-q}\\
			&\asymp
			\Big(\sum\limits_{B_i\in\mathfrak{D}_j}f(\diam B_i)\Big)^q
			\Big(\int\limits_{0}^{2^{-j}}\Big(\frac{t^\alpha}{f(t)}\Big)^{\frac{q}{1-q}}\frac{dt}{t}\Big)^{1-q}
			\overset{j\rightarrow\infty}{\longrightarrow}
			0.
		\end{aligned}
	\end{equation}

	{\bf Item 9} states that if $1<q\leqslant\infty$ and $f$ is a regular function such that
	\begin{equation}
		\int\limits_{0}^{1}\left(\frac{f(t)}{t^\alpha}\right)^{\frac{q}{q-1}}\frac{dt}{t}<\infty,
	\end{equation}
	then $\Lambda_f$ is absolutely continuous with respect to $\mathcal{H}_{\alpha,q}$.
	Moreover if $\Lambda_f(S)>0$, then $\mathcal{H}_{\alpha,q}(S)=\infty$; if $\mathcal{H}_{\alpha,q}(S)<\infty$, then $\Lambda_f(S)=0$.
	
	It suffices to prove that if $\Lambda_f(S)>0$, then $\mathcal{H}_{\alpha,q}(S)=\infty$.
	Let $\mathcal{H}_{\alpha,q}(S)<\infty$. It means that there exist coverings $\mathfrak{D}_j$ such that the diameters of all sets are not greater than $2^{-j}$ and satisfy the estimate
	\begin{equation}
		\sum\limits_{k=j}^{\infty}\Big(\underset{\diam B_i\in\Delta_k}{\sum\limits_{B_i\in\mathfrak{D}_j}}(\diam B_i)^{\alpha}\Big)^q\leqslant 
		C.
	\end{equation}
	Let $M_k$ be the number of sets $B_i$ such that $\diam B_i\in\Delta_k$.		
	So, we can estimate $\Lambda_f(S)$ by the quantity
	
	\begin{equation}
		\begin{aligned}
			\sum\limits_{B_i\in\mathfrak{D}_j}f(\diam B_i)
			&\asymp
			\sum\limits_{k=j}^\infty M_kf(2^{-k})
			\lesssim
			\Big(\sum\limits_{k=j}^\infty M_k^q2^{-qk\alpha}\Big)^{\frac{1}{q}}\Big(\sum\limits_{k=j}^\infty \Big(\frac{f(2^{-k})}{2^{-k\alpha}}\Big)^{\frac{q}{q-1}}\Big)^{\frac{q-1}{q}}\\
			&\asymp
			\Big(\sum\limits_{k=j}^{\infty}\Big(\underset{\diam B_i\in\Delta_k}{\sum\limits_{B_i\in\mathfrak{D}_j}}(\diam B_i)^{\alpha}\Big)^q\Big)^{\frac{1}{q}}
			\Big(\int\limits_{0}^{2^{-j}}\Big(\frac{f(t)}{t^\alpha}\Big)^{\frac{q}{q-1}}\frac{dt}{t}\Big)^{\frac{q-1}{q}}
			\overset{j\rightarrow\infty}{\longrightarrow}
			0.
		\end{aligned}
	\end{equation}

	\subsection*{Proof of Lemma~\ref{frost}.}

	To prove this lemma, we need some auxiliary results.
	
	\begin{Le}\label{Morse1}
		For any $d\in\mathbb{N}$ there exists a constant $\theta(d)$ such that the following holds. Let $A$ be a bounded subset of $\mathbb{R}^d$ and let $\mathfrak{B}$ be a family of balls in~$\mathbb{R}^d$.
		Assume that
		\begin{equation}
			\forall x\in A \ \ \ \exists B_r(y)\in \mathfrak{B} \ \ \  such \ that \ x\in B_{\frac{r}{3}}(y).
		\end{equation}
		Then there exists a subfamily $\mathfrak{B}'\subset\mathfrak{B}$ satisfying the requirements
		\begin{enumerate}[1)]
			\item $A\subset \underset{B_{r_i}(x_i)\in\mathfrak{B}'}{\bigcup}B_{r_i}(x_i)$,
			
			\item the family $\mathfrak{B}'$ may be split into $\theta$ disjoint subfamilies.
		\end{enumerate}
	\end{Le}
	Lemma~\ref{Morse1} is a particular case of the Morse covering theorem (see~p.6~in~\cite{Guzman1975}).

	\begin{Le}\label{pok1}
		Let $\alpha\in[0,d]$ and $q\in(0,\infty]$. There exist the constants $C(\alpha,q)$ and $\theta(d)$ such that for any bounded $A\subset\mathbb{R}^d$, $\mathcal{H}_{\alpha,q}(A)<a$ for some constant $a>0$, and for any $\varepsilon>0$ there exists a family of balls $\mathfrak{B}$ such that
		\begin{enumerate}[1)]
			\item $\mathfrak{B}$ is a covering of $A$,
			\item the center of any ball in~$\mathfrak{B}$ lies in~$A$ and its radius does not exceed~$\varepsilon$,
			
			\item the family~$\mathfrak{B}$ may be split into~$\theta$ disjoint subfamilies~$\mathfrak{B}^j$ such that
			\begin{equation}
				\|\{r_i| \ {B_{r_i}(x_i)\in\mathfrak{B}^j}\}\|_{\ell_{\alpha,\alpha q}}^{\alpha q}\leqslant Ca,
			\end{equation}
			in the case $q<\infty$ or 
			\begin{equation}
				\|\{r_i| \ {B_{r_i}(x_i)\in\mathfrak{B}^j}\}\|_{\ell_{\alpha,\infty}}^{\alpha}\leqslant Ca
			\end{equation}
			in the case $q=\infty$.
		\end{enumerate}
	\end{Le}
	This lemma for classical Hausdorff measure was proved in~\cite[Lemma~3.4]{My}.
	\begin{proof}
		We present the proof in the case $q<\infty$. The proof in the case $q=\infty$ is similar.
		
		Since $\mathcal{H}_{\alpha,q}(A)<a$ there exists a family $\mathfrak{A}$ such that $\mathfrak{A}$ is a covering of $A$, for any $D\in\mathfrak{A}$ we have $\diam D<\frac{\eps}{1000}$, and
		\begin{equation}
			\|\{\diam D_i|\ D_i\in\mathfrak{A}\}\|_{\ell_{\alpha,\alpha q}}^{\alpha q}\lesssim a.
		\end{equation}
		
		The family $\mathfrak{B}_0$ is defined in the following way: for any $D\in \mathfrak{A}$ we choose a point $x\in D\cap A$ and put the ball $B_{3\diam D}(x)$ into $\mathfrak{B}_0$. The family $\mathfrak{B}_0$ and the set $A$ satisfy the conditions of Lemma~\ref{Morse1}. Let $\mathfrak{B}$ be the subfamily of $\mathfrak{B}_0$ provided by Lemma~\ref{Morse1}. The family $\mathfrak{B}$ splits into $\theta(d)$ disjoint subfamilies $\mathfrak{B}^j$ that satisfy the estimate:
		\begin{equation}
			\|\{r_i|\ B_{r_i}(x_i)\in\mathfrak{B}^j\}\|_{\ell_{\alpha,\alpha q}}^{\alpha q}\leqslant
			\|\{r_i|\ B_{r_i}(x_i)\in\mathfrak{B}_0\}\|_{\ell_{\alpha,\alpha q}}^{\alpha q}\leqslant
			\|\{3\diam D_i|\ D_i\in\mathfrak{A}\}\|_{\ell_{\alpha,\alpha q}}^{\alpha q}\leqslant
			3^{\alpha q}a.
		\end{equation}
		
	\end{proof}

	\begin{Le}\label{PPP}
		Let $\mu$ be a signed measure, let $A_{+}$ and $A_{-}$ be the sets of its Hahn decomposition, let $\mu_{+}$ and $\mu_{-}$ be its positive and negative parts. Consider the set 
		\begin{equation}
			P_{+,\varepsilon}=\Big\{x\in A_{+}\mid\exists \delta(x) \ such \ that \  \forall r<\delta(x) \ \  \mu_{-}(B_r(x))\leqslant\varepsilon\mu_{+}(B_r(x))\Big\}.
		\end{equation}
		Then $\mu_{+}(A)=\mu_{+}(P_{+,\varepsilon})$.
		
	\end{Le}
	See~\cite{StolyarovWojciechowski2014} for the proof of a similar lemma (Lemma 4 of that paper). Consider the set $P_{+,\varepsilon}^{(N)}$ given by the formula
	\begin{equation}
		P_{+,\varepsilon}^{(N)}=\Big\{x\in A_{+}\mid\  \forall r<\frac{1}{N} \ \  \mu_{-}(B_r(x))\leqslant\varepsilon\mu_{+}(B_r(x))\Big\}.
	\end{equation}
	
	\begin{Le}
		Let $x\in P_{+,\varepsilon}^{(N)}$ and let $r<\frac{1}{3N}$. Then
		\begin{equation}
			\int \varphi\left(\frac{y-x}{r} \right)d\mu_{-}(y)\leqslant \varepsilon\int \varphi\left(\frac{y-x}{r} \right)d\mu_{+}(y)
		\end{equation}
		for any radial non-negative test-function $\varphi$ supported in $B_3(0)$ that decreases as the radius grows.
	\end{Le}
	\begin{proof}
		The statement of the lemma is trivial in the case $\varphi$ is a linear combination of characteristic functions of balls. An arbitrary function $\varphi$ can be obtained as a uniform limit of functions of this type.
	\end{proof}
	\begin{Cor}\label{noc}
		Let $x\in P_{+,\varepsilon}^{(N)}$ and let $r<\frac{1}{3N}$. Then
		\begin{equation}
			\int \varphi\left(\frac{y-x}{r} \right)d\mu_{+}(y)\leqslant \frac{1}{1-\varepsilon}\int \varphi\left(\frac{y-x}{r} \right)d\mu(y)
		\end{equation}
		for any radial non-negative test-function $\varphi$ supported in $B_3(0)$ that decreases as the radius grows.
	\end{Cor}
	\begin{proof}[Prove Lemma~\ref{frost}]
		Suppose that $A$ is a set such that $\mathcal{H}_{\alpha,q}(A)<a$, we will prove that $|\mu|(A)\lesssim a^\gamma$. Due to Lemma~\ref{PPP} it suffices to prove that $\mu_+(P_{+,\frac{1}{2}})\lesssim a^\gamma$, this provides the bound for $\mu_+(A)$. The bound for $\mu_-(A)$ is similar. Note that $P_{+,\frac{1}{2}}=\cup P_{+,\frac{1}{2}}^{(N)}$, so $\mu_{+}(P_{+,\frac{1}{2}})\leqslant 2\mu_{+}(P_{+,\frac{1}{2}}^{(N)})$ for $N$ sufficiently large. Let $K$ be a compact subset of $P_{+,\frac{1}{2}}^{(N)}$ such that $\mu_{+}(P_{+,\frac{1}{2}}^{(N)})\leqslant2\mu_{+}(K)$. We will prove that $\mu_{+}(K)\lesssim a^\gamma$.
		Let $\mathfrak{B}$ be a covering of $K$ provided by Lemma~\ref{pok1} such that the radius of any ball does not exceed $\frac{1}{3N}$ ($K$ is compact set, so we may assume that $\mathfrak{B}$ is finite). We can write (in the case~$q<\infty$)
		\begin{equation}
			\begin{aligned}
				\mu_{+}(K)&\lesssim
				\sum_{B_{r_i}(x_i)\in\mathfrak{B}}\int\limits_{B_{r_i}(x_i)}\varphi\left(\frac{y-x_i}{r_i}\right)d\mu_{+}(y)\leqslant
				\sum_{B_{r_i}(x_i)\in\mathfrak{B}}\int\limits_{B_{3r_i}(x_i)}\varphi\left(\frac{y-x_i}{r_i}\right)d\mu_{+}(y)\\
				&\!\!\!\!\!\overset{\mathrm{Cor}~\ref{noc}}{\lesssim}
				\sum_{B_{r_i}(x_i)\in\mathfrak{B}}\int\limits_{B_{3r_i}(x_i)}\varphi\left(\frac{y-x_i}{r_i}\right)d\mu(y)\lesssim
				\sum_{j=1}^{M}\sum_{B_{r_i}(x_i)\in\mathfrak{B}^j}\int\limits_{\mathbb{R}^d}\varphi\left(\frac{y-x_i}{r_i}\right)d\mu(y)\\
				&\lesssim 
				\sum_{j=1}^{M}\left(\|\{r_i|\ B_{r_i}(x_i)\in\mathfrak{B}^j\}\|_{\ell_{\alpha,\alpha q}}^{\alpha q}\right)^\gamma\lesssim
				a^\gamma.
			\end{aligned}
		\end{equation}
		
		The case $q=\infty$ is similar.
	\end{proof}

		\subsection*{Proof of Lemma~\ref{strash}.}
	
	Fix a family of disjoint balls $\mathfrak{B}$. Let $\{r_j\}$ be the sequence of radii of the balls in the family $\mathfrak{B}$. Let $\ell_{p,h}(\mathfrak{B},\alpha)$ be the Lorentz space on sequences with the weight $r_j^\alpha$, we mean that $\ell_{p,h}(\mathfrak{B},\alpha)=L_{p,h}(\mathbb{N},\nu_\alpha)$  there $\nu_{\alpha}\{j\}=r_j^\alpha$. Let $\{e_j\}$ be the standard basis in the sequence space and let the operator $T$ act by the rule
	\begin{equation}
		T(e_j)=\varphi_{x_j,r_j}.
	\end{equation}
	By Theorem~\ref{SW} $T$ is continuous from $\ell_{p,1}(\mathfrak{B},\alpha_0)$ to $X_0$ and from $\ell_{p,1}(\mathfrak{B},\alpha_1)$ to $X_1$. 
		
	To prove the lemma it suffices to obtain the estimate
	\begin{equation}
		\|\chi_{\mathbb{N}}\|_{(\ell_{p,1}(\mathfrak{B},\alpha_0),\ell_{p,1}(\mathfrak{B},\alpha_1))_{\theta,h}}\lesssim
		\|\{r_j| \ {B_{r_j}(x_j)\in \mathfrak{B}}\}\|^{\frac{\alpha_{\theta}}{p}}_{\ell_{\alpha_\theta,\frac{h\alpha_\theta}{p}}}.
	\end{equation}
	Firstly we prove the inequality
	\begin{equation}
		\|\chi_{\mathbb{N}}\|_{(\ell_{p,1}(\mathfrak{B},\alpha_0),\ell_{p,1}(\mathfrak{B},\alpha_1))_{\theta,h}}\approx
		\|\chi_{\mathbb{N}}\|_{(\ell_{p}(\mathfrak{B},\alpha_0),\ell_{p}(\mathfrak{B},\alpha_1))_{\theta,h}}.
	\end{equation}
	This inequality is based on the fact that for characteristic functions their norms in the spaces $L_{p,1}$ and $L_p$ are equivalent. From this fact it is easy to see that
	\begin{equation}
		K(\chi_{\mathbb{N}},t,\ell_{p,1}(\mathfrak{B},\alpha_0),\ell_{p,1}(\mathfrak{B},\alpha_1))\approx K(\chi_{\mathbb{N}},t,\ell_{p}(\mathfrak{B},\alpha_0),\ell_{p}(\mathfrak{B},\alpha_1)).
	\end{equation}
	The equivalence of $K$ functionals provides us with the equivalence of the interpolation norms. Let $\nu=\nu_{\alpha_0}$. Theorem~3.7 from~\cite{Gilbert} gives us the equivalent norm
	\begin{equation}
		\|\chi_{\mathbb{N}}\|_{(\ell_{p}(\mathfrak{B},\alpha_0),\ell_{p}(\mathfrak{B},\alpha_1))_{\theta,h}}
		=
		\Big(\int\limits_0^\infty\Big(t^{1-\theta}\Big(\int\limits_{\scalebox{0.6}{$r_j^{\frac{\alpha_1-\alpha_0}{p}}<t^{-1}$}}r_j^{\alpha_1-\alpha_0}d\nu\Big)^{\frac{1}{p}}\Big)^h\frac{dt}{t}\Big)^{\frac{1}{h}}.
	\end{equation}
	
	The following estimate finishes the proof of the lemma:
	\begin{equation}
		\begin{aligned}
			&\Big(\int\limits_0^\infty\Big(t^{1-\theta}\Big(\int\limits_{\scalebox{0.6}{$r_j^{\frac{\alpha_1-\alpha_0}{p}}<t^{-1}$}}r_j^{\alpha_1-\alpha_0}d\nu\Big)^{\frac{1}{p}}\Big)^h\frac{dt}{t}\Big)^{\frac{1}{h}}
			=
			\Big(\int\limits_0^\infty\Big(t^{1-\theta}\Big(\sum\limits_{\scalebox{0.6}{$r_j>t^{\frac{p}{\alpha_0-\alpha_1}}$}}r_j^{\alpha_1}\Big)^{\frac{1}{p}}\Big)^h\frac{dt}{t}\Big)^{\frac{1}{h}}\\
			&\leqslant
			\Big(\int\limits_0^\infty\big(t^{1-\theta}\big(m_{\mathfrak{B}}(t^{\frac{p}{\alpha_0-\alpha_1}})t^{\frac{\alpha_1p}{\alpha_0-\alpha_1}}\big)^{\frac{1}{p}}\big)^h\frac{dt}{t}\Big)^{\frac{1}{h}}
			=
			\Big(\frac{\alpha_0-\alpha_1}{p}\int\limits_0^\infty \big(\tau^{(1-\theta)\frac{\alpha_0-\alpha_1}{p}}m_{\mathfrak{B}}(\tau)^{\frac{1}{p}}\tau^{\frac{\alpha_1}{p}}\big)^h\frac{d\tau}{\tau}\Big)^{\frac{1}{h}}\\
			&=
			\Big(\frac{\alpha_0-\alpha_1}{p}\int\limits_0^\infty \big(\tau^{\alpha_{\theta}}m_{\mathfrak{B}}(\tau)\big)^\frac{h}{p}\frac{d\tau}{\tau}\Big)^{\frac{1}{h}}
			\approx
			\|\{r_j| \ {B_{r_j}(x_j)\in \mathfrak{B}}\}\|^{\frac{\alpha_{\theta}}{p}}_{\ell_{\alpha_\theta,\frac{h\alpha_\theta}{p}}}.
		\end{aligned}
	\end{equation}

	\subsection*{Proof of Lemma~\ref{p+p}}

	To prove this lemma we need an auxiliary result.

	\begin{Le}\label{tr}
		Let $\varepsilon>0$. There exists a constant $C_\varepsilon$ such that the inequality
		\begin{equation}
			\|f+g\|_{L_{p,q}}\leqslant (1+\varepsilon)\|f\|_{L_{p,q}}+C_\varepsilon\|g\|_{L_{p,q}}
		\end{equation}
		is true.
	\end{Le}
	The case $1\leqslant p=q$ is simple.
	\begin{proof}
		We use the following inequalities for any $0<\delta<1$
		\begin{equation}
			m_{f+g}(t)\leqslant m_f((1-\delta)t)+m_g(\delta t);
		\end{equation}
		for every $r>0$ and $\alpha>0$ there exists a constant $C_{r,\alpha}$ such that the inequality 
		\begin{equation}
			|a+b|^r\leqslant (1+\alpha)|a|^r+C_{r,\alpha}|b|^r.
		\end{equation}
		We can write the estimate
		\begin{equation}
			\begin{aligned}
				\|f+g\|^q_{L_{p,q}}
				&=
				q\int_0^\infty (m_{f+g}(t)t^p)^{\frac{q}{p}}\frac{dt}{t}
				\leqslant
				q\int_0^\infty \big((m_{f}((1-\delta)t)+m_g(\delta t))t^p\big)^{\frac{q}{p}}\frac{dt}{t}\\
				&\leqslant
				q\int_0^\infty (1+\delta)(m_{f}((1-\delta)t)t^p)^{\frac{q}{p}}\frac{dt}{t}+
				q\int_0^\infty C_{\frac{q}{p},\delta}(m_{g}(\delta t)t^p)^{\frac{q}{p}}\frac{dt}{t}\\
				&=
				(1+\delta)(1-\delta)^{-q}\|f\|^q_{L_{p,q}}+C_{\frac{q}{p},\delta}\delta^{-q}\|g\|^q_{L_{p,q}}.
			\end{aligned}
		\end{equation}
		The following estimate finishes the proof
		\begin{equation}
			\begin{aligned}
				\|f+g\|_{L_{p,q}}
				&\leqslant \left((1+\delta)(1-\delta)^{-q}\|f\|^q_{L_{p,q}}+C_{\frac{q}{p},\delta}\delta^{-q}\|g\|^q_{L_{p,q}}\right)^{\frac{1}{q}}\\
				&\leqslant
				(1+\delta)^{1+\frac{1}{q}}(1-\delta)^{-1}\|f\|_{L_{p,q}}+C_{\frac{1}{q},\delta}C_{\frac{q}{p},\delta}^{\frac{1}{q}}\delta^{-1}\|g\|_{L_{p,q}}.
			\end{aligned}
		\end{equation}
	\end{proof}

	\begin{proof}[Proof of Lemma~\ref{p+p}]
		For some function $h$ the functions $h_{\geqslant\varepsilon}$ and $h_{<\varepsilon}$ are defined by formulas
		\begin{equation}
			h_{\geqslant\varepsilon}=\chi_{\{|h|\geqslant\varepsilon\}}h,
		\end{equation}
		\begin{equation}
			h_{<\varepsilon}=\chi_{\{|h|<\varepsilon\}}h.
		\end{equation}
		
		It is easy to see that $h=h_{\geqslant\varepsilon}+h_{<\varepsilon}$. We can write the inequality
		\begin{equation}\label{texv}
			\begin{aligned}
				&\|f+g_j\|_{L_{p,q}}
				\overset{\scriptscriptstyle{\text{Lem~\ref{tr}}}}{\leqslant}
				 (1+\alpha)\|f_{\geqslant\varepsilon}+{(g_j)}_{<\delta}\chi_{\{f< \varepsilon\}}\|_{L_{p,q}}+C_\alpha\|{g_j}_{\geqslant\delta}+f_{<\varepsilon}+{(g_j)}_{<\delta}\chi_{\{f\geqslant \varepsilon\}}\|_{L_{p,q}}\\
				&\leqslant
				(1+\alpha)\|f_{\geqslant\varepsilon}+{(g_j)}_{<\delta}\chi_{\{f< \varepsilon\}}\|_{L_{p,q}}+
				CC_\alpha\big(\|{(g_j)}_{\geqslant\delta}\|_{L_{p,q}}+\|f_{<\varepsilon}\|_{L_{p,q}}+\|{(g_j)}_{<\delta}\chi_{\{f\geqslant \varepsilon\}}\|_{L_{p,q}}\big).
			\end{aligned}
		\end{equation}
		
		Let $\delta<\varepsilon$, then 
		
		\begin{equation}
			(f_{\geqslant\varepsilon}+{(g_j)}_{<\delta}\chi_{\{f< \varepsilon\}})^*(t)=
			\begin{cases}
				f^*(t), &   t\leqslant m_f(\varepsilon);\\
				({(g_j)}_{<\delta}\chi_{\{f< \varepsilon\}})^*(t-m_f(\varepsilon)), & t>m_f(\varepsilon).
			\end{cases}
		\end{equation}
		
		So we can write
		\begin{equation}
			\begin{aligned}
				&\|f_{\geqslant\varepsilon}+{(g_j)}_{<\delta}\chi_{\{f< \varepsilon\}}\|^q_{L_{p,q}}=
				\frac{q}{p}\int_{0}^{m_f(\varepsilon)}f^*(t)^qt^{\frac{q}{p}}\frac{dt}{t}+
				\frac{q}{p}\int_{m_f(\varepsilon)}^{\infty}({(g_j)}_{<\delta}\chi_{\{f< \varepsilon\}})^*(t-m_f(\varepsilon))^qt^{\frac{q}{p}}\frac{dt}{t}\\
				&\leqslant
				\|f\|^{q}_{L_{p,q}}+
				\frac{q}{p}\int_{m_f(\varepsilon)}^{Nm_f(\varepsilon)}\delta^qt^{\frac{q}{p}}\frac{dt}{t}+
				\frac{q}{p}\int_{Nm_f(\varepsilon)}^{\infty}g_j^*(t-m_f(\varepsilon))^qt^{\frac{q}{p}}\frac{dt}{t}\\
				&\leqslant
				\|f\|^q_{L_{p,q}}+\delta^{q}(Nm_f(\varepsilon))^{\frac{q}{p}}+
				\frac{q}{p}\int_{(N-1)m_f(\varepsilon)}^{\infty}g_j^*(t)^q\max\{1,(1+\frac{1}{N-1})^{\frac{q}{p}-1}\}t^{\frac{q}{p}}\frac{dt}{t}\\ 
				&\leqslant
				\|f\|^q_{L_{p,q}}+\max\{1,(1+\frac{1}{N-1})^{\frac{q}{p}-1}\}\|g_j\|^q_{L_{p,q}}+\delta^q(Nm_f(\varepsilon))^{\frac{q}{p}}.
			\end{aligned}
		\end{equation}
		
		Therefore, we have
		\begin{equation}\label{dolim}
			\begin{aligned}
				\|f+g_j\|_{L_{p,q}}
				&\leqslant (1+\alpha)\left(\|f\|^q_{L_{p,q}}+\max\{1,(1+\frac{1}{N-1})^{\frac{q}{p}-1}\}\|g_j\|^q_{L_{p,q}}+\delta^q(Nm_f(\varepsilon))^{\frac{q}{p}}\right)^{\frac{q}{p}}\\
				&+
				CC_\alpha(\|{(g_j)}_{\geqslant\delta}\|_{L_{p,q}}+\|f_{<\varepsilon}\|_{L_{p,q}}+\|{(g_j)}_{<\delta}\chi_{\{f\geqslant \varepsilon\}}\|_{L_{p,q}}).
			\end{aligned}
		\end{equation}

		We have $\lim\|g_j\|_{L_{p_1}}=0$, consequently $\lim\|{(g_j)}_{\geqslant\delta}\|_{L_{p,q}}=0$. We pass to the upper limit in~\eqref{dolim}:
		\begin{equation}
			\begin{aligned}
			\limsup \|f+g_j\|_{L_{p,q}}
			&\leqslant
			(1+\alpha)\left(\|f\|^q_{L_{p,q}}+\max\{1,(1+\frac{1}{N-1})^{\frac{q}{p}-1}\}A^q+\delta^q(Nm_f(\varepsilon))^{\frac{q}{p}}\right)^{\frac{1}{q}}\\
			&+
			CC_\alpha(\|f_{<\varepsilon}\|_{L_{p,q}}+\delta\|\chi_{\{f\geqslant \varepsilon\}}\|_{L_{p,q}}).
			\end{aligned}
		\end{equation}
		
		Considering arbitrarily small $\delta$, we get
		\begin{equation}
			\limsup \|f+g_j\|_{L_{p,q}}\leqslant
			(1+\alpha)\left(\|f\|^q_{L_{p,q}}+\max\{1,(1+\frac{1}{N-1})^{\frac{q}{p}-1}\}A^q\right)^{\frac{1}{q}}+
			CC_\alpha\|f_{<\varepsilon}\|_{L_{p,q}}.
		\end{equation}
		
		Let $\varepsilon\rightarrow 0$, then 
		\begin{equation}
			\limsup \|f+g_j\|_{L_{p,q}}\leqslant
			(1+\alpha)\left(\|f\|^q_{L_{p,q}}+\max\{1,(1+\frac{1}{N-1})^{\frac{q}{p}-1}\}A^q\right)^{\frac{1}{q}}.
		\end{equation}
		
		Finally, let $\alpha\rightarrow 0$ and $N\rightarrow\infty$.  Then, we have
		\begin{equation}
			\limsup \|f+g_j\|_{L_{p,q}}\leqslant
			\left(\|f\|^q_{L_{p,q}}+A^q\right)^{\frac{1}{q}}.
		\end{equation}

	\end{proof}

	\subsection*{Auxiliary lemma}
	
	\begin{Le}\label{Lornor}
		Let $0<\alpha<\infty$ and $0<q<\infty$. Let $a=(a_1,a_2,\dots)$ be a sequence of real numbers. Then
		\begin{equation}
			\Big(\sum\limits_{k}\Big(\sum\limits_{|a_j|\in\Delta_k}|a_j|^\alpha\Big)^q\Big)^{\frac{1}{q\alpha}}
			\asymp
			\|a\|_{l_{\alpha,\alpha q}},
		\end{equation}
		\begin{equation}
		\sup\limits_{k}	\Big(\sum\limits_{|a_j|\in\Delta_k}|a_j|^\alpha\Big)^{\frac{1}{\alpha}}
			\asymp
			\|a\|_{l_{\alpha,\infty}}.
		\end{equation}
	\end{Le}
	\begin{proof}
		Let $b_k=m_a(2^{-k-1})-m_a(2^{-k})$, then we have
		\begin{equation}
			\sum\limits_{k}\Big(\sum\limits_{|a_j|\in\Delta_k}|a_j|^\alpha\Big)^q
			\asymp
			\sum\limits_{k} b_k^q 2^{-\alpha qk},
		\end{equation}
		\begin{equation}
			\sup\limits_{k}\sum\limits_{|a_j|\in\Delta_k}|a_j|^\alpha
			\asymp
			\sup\limits_{k}( b_k 2^{-\alpha k}).
		\end{equation}
		We can write the estimates
		\begin{equation}
			\|a\|_{l_{\alpha,\alpha q}}^{\alpha q}
			=
			\alpha q\int\limits_{0}^{\infty}m_a(t)^q t^{\alpha q}\frac{dt}{t}
			\asymp
			\sum\limits_{k} ^q m_a(2^{-k-1})2^{-\alpha qk}
			=
			\sum\limits_{k} \Big(\sum\limits_{j=-\infty}^{k}b_j\Big)^q 2^{-\alpha qk},
		\end{equation}
		\begin{equation}
			\|a\|_{l_{\alpha,\infty}}^{\alpha}
			=
			\sup m_a(t) t^{\alpha}
			\asymp
			\sup  m_a(2^{-k-1})2^{-\alpha k}
			=
			\sup \Big(\sum\limits_{j=-\infty}^{k}b_j\Big) 2^{-\alpha k}.
		\end{equation}
		
		Now we consider several cases
		
		Case $q=\infty$:
		\begin{equation}
			\begin{aligned}
				\sup\limits_{k}( b_k 2^{-\alpha k})
				&\lesssim
				\sup \Big(\sum\limits_{j=-\infty}^{k}b_j\Big) 2^{-\alpha k}\\
				&\lesssim
				\sup \Big(\sum\limits_{j=-\infty}^{k}\sup\limits_{l}( b_l 2^{-\alpha l})2^{j\alpha}\Big) 2^{-\alpha k}
				=
				\sup\limits_{k}( b_k 2^{-\alpha k})\sum\limits_{j=0}^{\infty}2^{-\alpha j}
				\asymp
				\sup\limits_{k}( b_k 2^{-\alpha k}).
			\end{aligned}
		\end{equation}
		
		Case $q\leqslant 1$:
		\begin{equation}
			\begin{aligned}
				\sum\limits_{k} b_k^q 2^{-\alpha qk}
				&\lesssim
				\sum\limits_{k} \Big(\sum\limits_{j=-\infty}^{k}b_j\Big)^q 2^{-\alpha qk}\\
				&\lesssim
				\sum\limits_{k} \sum\limits_{j=-\infty}^{k}b_j^q 2^{-\alpha qk}
				=
				\sum\limits_{k}b_k^q 2^{-\alpha qk} \sum\limits_{j=0}^{\infty}2^{-\alpha qj}
				\asymp
				\sum\limits_{k} b_k^q 2^{-\alpha qk}.
			\end{aligned}
		\end{equation}
		
		Case $1<q<\infty$:
		\begin{equation}
			\begin{aligned}
				\sum\limits_{k} b_k^q 2^{-\alpha qk}
				&\lesssim
				\sum\limits_{k} \Big(\sum\limits_{j=-\infty}^{k}b_j\Big)^q 2^{-\alpha qk}\\
				&\lesssim
				\sum\limits_{k} \Big(\sum\limits_{j=-\infty}^{k}b_j^q 2^{\frac{\alpha q(k-j)}{2}}\Big)\Big(\sum\limits_{j=-\infty}^{k} 2^{-\frac{\alpha q'(k-j)}{2}}\Big)^{\frac{q}{q'}} 2^{-\alpha qk}\\
				&=
				\sum\limits_{k}b_k^q 2^{-\alpha qk} \sum\limits_{j=0}^{\infty}2^{\frac{-\alpha qj}{2}}\Big(\sum\limits_{j=0}^{\infty} 2^{-\frac{\alpha q'j}{2}}\Big)^{\frac{q}{q'}}
				\asymp
				\sum\limits_{k} b_k^q 2^{-\alpha qk}.
			\end{aligned}
		\end{equation}
	\end{proof}
	
	\section{Some generalizations and an open question}\label{genops}

	\begin{Def}
		Let $\Phi$ be a regular function and $\alpha>0$ be a real number. Define the set function $\mathcal{H}_{\alpha,\Phi}$ by the formula
		\begin{equation}			
			\mathcal{H}_{\alpha,\Phi}(F,\delta)=\underset{\diam B_j<\delta}{\underset{F\subset\cup B_j}{\inf}}\sum\limits_{k}\Phi\left(\sum\limits_{\diam B_j\in\Delta_k}(\diam B_j)^\alpha\right).
		\end{equation}
	\end{Def}
	Note that if $\Phi(t)=t^{q}$ then $\mathcal{H}_{\alpha,\Phi}=\mathcal{H}_{\alpha,q}$. One can notice that our construction works in the general case. 
	\begin{Th}
		Let $2<p<\infty$, $1\leqslant q<\infty$. Let $\Phi$ be a regular function such that
		\begin{equation}
			\frac{\Phi(t)}{t^{\frac{q'}{2}}}\overset{t\rightarrow0}{\longrightarrow}0.
		\end{equation}
		Then there exists a compact set $S\in \mathbb{R}^d$ and a probability measure $\mu$ such that $\supp\mu\subset S$, $\hat{\mu}\in L_{p,q}(\mathbb{R}^d)$ and $\mathcal{H}_{\frac{2d}{p},\Phi}(S)=0$.
	\end{Th}
	The proof is the same as the proof of Theorem~\ref{resL}. The difference is that we replace equation~\eqref{soot} with
	\begin{equation}
		r_k=\left(\frac{\Phi^{-1}\Big(\frac{b(Q_k)}{n(Q_k)+1}\Big)}{M_k}\right)^{\frac{p}{2d}}.
	\end{equation}
	
	\begin{Que}
		Let $p>2$ and $q>0$. For what functions $f$ does there exist a compact set $S\subset\mathbb{R}^d$ and a measure $\mu$ such that $\supp \mu\subset S$, $\hat{\mu}\in L_{p,q}$ and $\Lambda_f(S)=0$?
	\end{Que}
	Theorem~\ref{merm} says that for $q\leqslant 2$ and $f(t)=t^{\frac{2d}{p}}$ there does not exist such a compact set and a measure. If $g(t)/t^{\frac{2d}{p}}\rightarrow 0$ and the sequence $\{M_j\}$ grows sufficiently fast then the compact set $S$, which we constructed in Section~\ref{S3}, satisfies $\Lambda_g(S)=0$. For $q>2$ the question is open. Theorem~\ref{resL} says that if $f(t)=t^{\frac{2d}{p}}$ then such a compact set and a measure exist. From this we can conclude that for some function $g$ such that $g(t)/t^{\frac{2d}{p}}\rightarrow\infty$ we have $\Lambda_g(S)=0$. But how fast should this function tend to zero?

	\
	
	\
	
	\
	
	{\small
		
		St. Petersburg State University Department of Mathematics and Computer Sciences;\\
		e-mail: n.dobronravov@spbu.ru\\
		\bigskip
		
	}
	
\end{document}